\documentclass[12pt, a4paper,reqno]{amsart}
\usepackage[a4paper, margin=4cm]{}
\usepackage{amssymb,amsfonts, amsthm, xfrac, amscd}
\usepackage{setspace}
\usepackage{breqn}
\usepackage{times}
\usepackage{changes}
\usepackage{adjustbox}
\usepackage{color}
\usepackage{array}
\usepackage{parskip}
\usepackage{microtype}

\usepackage{color}
\usepackage[hidelinks]{hyperref}

\allowdisplaybreaks[4]

\newtheoremstyle{myremark}     {11pt}{11pt}{}{}{\bfseries}{.}{.5em}{}

\newtheorem{thm}{Theorem}[section]
\newtheorem{cor}[thm]{Corollary}
\newtheorem{lem}[thm]{Lemma}
\newtheorem{pro}[thm]{Proposition}
\theoremstyle{definition}
\newtheorem{defn}[thm]{Definition}
\newtheorem{exmp}[thm]{Example}
\theoremstyle{myremark}
\newtheorem{rem}[thm]{Remark}

\numberwithin{equation}{section}
\begin{document}

	\title[Linear combination of bilateral gamma random variables]{Linear combination of bilateral gamma random variables: distributional theory and approximations}

	\author[Barman and  Vellaisamy]{Kalyan Barman and  Palaniappan Vellaisamy}
	\address{\hskip-\parindent
		Kalyan Barman, Department of Mathematics, NIT Warangal,
		Warangal - 506004, India.}
	
	\email{barmankalyan@nitw.ac.in}



		\address{\hskip-\parindent
		P Vellaisamy, Department of Statistics and Applied Probability, UC Santa Barbara, Santa Barbara, CA, 93106, USA.}
	\email{pvellais@ucsb.edu}
	
	\begin{abstract}
		In this article, we obtain the exact distribution of a linear combination of bilateral gamma (BG) random variables (r.v.s). Next, we discuss the distributional properties of the linear combination of BG r.v.s, including probability density function, cumulant generating function and characteristic function. A Stein characterization is developed, which leads us to several distributional approximation results with explicit error bounds in both Kolmogorov and Wasserstein distances. Related limit theorems are also discussed. Furthermore, we show that the associated L\'evy processes are finite-variation processes with BG distributed increments having random parameters. Finally, we apply our results in exponential stock models.
	\end{abstract}

	\maketitle
	
\section{Introduction} \label{Intro}	
\noindent
The distribution of linear combination of independent gamma random variables (r.v.s) emerge naturally in various areas of probability and statistics, see Diaconis and Perlman \cite{DP1990}. The authors mentioned that the distribution of such a combination is not expressible in a closed form and so discussed approximations and studied tail probabilities. Exploiting the relationship between  negative binomial and gamma distributions (see \cite{EZ1980} or \cite{VS2010}), Vellaisamy and Upadhye \cite{pvns2009} derived the exact distribution of a linear combination of independent gamma r.v.s with arbitrary parameters.

\noindent
The bilateral-gamma (BG) distributions form a four-parameter family with a rich distributional structure (see \cite{bilateral0} and \cite{bilateral1}) and the properties that are useful for applications. This family includes the variance gamma (VG), symmetric variance gamma (SVG), and Laplace distributions, as special cases. In addition,  gamma and normal distributions appear as limiting cases, together with several other related distributions. The BG distributions are self-decomposable, stable under convolution, and have a simple cumulant generating function (cgf). The associated L\'evy processes are finite-variation processes making infinitely many jumps at each interval with positive length, and  their increments are BG distributed. Hence, the BG process provides a flexible framework for modelling fluctuations in financial markets, see K$\ddot{u}$chler and Tappe \cite{bilateral2}. 

\noindent
In this paper, we investigate the distribution of linear combinations of BG r.v.s. This family includes the linear combinations of gamma r.v.s,  as  limiting cases,  studied by Vellaisamy and Upadhye \cite{pvns2009}. We discuss several probablistic properties  of linear combinations of BG r.v.s, and show how they are related to the other distributions studied in the literature. They possess a number  of properties which make them useful for applications. The distribution of a linear combination of BG r.v.s is infinitely divisible, self-decomposable, stable under convolution, and has probability density function (pdf), cgf and characteristic function (cf) in closed form. We  derive also a Stein characterization for these distributions, obtain several approximation results using Stein's method, and discuss  related limit theorems. For the Kolmogorov distance, the error bounds for a sequence of compound Poisson distribution that converges to the distribution of a linear combination of BG r.v.s is ponited out.  For the Wasserstein distance, an error bound on the distance between  two such linear combinations is obtained. Finally, following K$\ddot{u}$chler and Tappe \cite{bilateral2}, we show that the associated L\'evy processes are finite-variation processes having infinitely many jumps at each interval with positive length, and  their increments are BG distributed, but with random parameters. 

\noindent
The paper is organized as follows. In Section \ref{nopre}, we discuss some preliminary results and in   Section \ref{probgdistribution},  the exact distribution of the linear combination of BG r.v.s  and their distributional properties are derived. Section \ref{BSBGD} deals with the approximation of linear combination of BG r.v.s. In Section \ref{BGpro}, we discuss, following K$\ddot{\text{u}}$chler and Tappe \cite{bilateral2},  the BG process associated
with  the linear combination of BG r.v.s and apply it to the exponential stock models.

\section{Preliminary results}\label{nopre}
\noindent
In this section, we introduce the notations and preliminaries required for this article.
\subsection{Bilateral-gamma distributions}\label{prebg}
We start with the definition and related results of BG distributions. Let $Ga(\alpha,p)$ denote the gamma distribution with parameters $\alpha$ and $p$. A BG distribution (also known as gamma difference distribution) with parameters $\alpha,\beta,p,q>0$ is defined as the distribution of $X=X_1-X_2$, where $X_1\sim Ga(\alpha,p)$ and $X_2\sim Ga(\beta,q)$ and they are independent, see \cite{gamdiff}. The cf of a BG distribution is 
\begin{align}
\phi_{bg}(z)&= \frac{1}{(1 -iz/\alpha)^p (1 +iz/\beta)^q} ,~~z\in\mathbb{R}\label{bicf0}\\
&=\exp\left(\int_{\mathbb{R}}(e^{izu}-1)\nu_{bg}(du)  \right),~~z\in\mathbb{R},\label{e1}
\end{align}
\noindent
where $\nu_{bg}$ is the L\'evy measure given by
\begin{align}\label{bglevymeasure}
\nu_{bg}(du)=\left(\frac{p  }{u}e^{-\alpha u}\mathbf{1}_{(0,\infty)}(u)-\frac{q  }{u}e^{-\beta|u|}\mathbf{1}_{(-\infty,0)}(u)\right)du.
\end{align}
\noindent
Taking Fourier transform (FT) on \eqref{bicf0}, the pdf of the $BG(\alpha,p,\beta,q)$ distribution has the integral form (see Equation 1.4 of \cite{gamdiff})
\begin{align}
h_X(x)=\frac{1}{2\pi}\displaystyle\int_{\mathbb{R}}\frac{e^{-ixz}}{(1-iz/\alpha)^{p} (1+iz/\beta)^{q}}dz.
\end{align}

\noindent
An alternative form of the pdf of the $BG(\alpha,p,\beta,q)$ distribution follows from the convolution structure (see Equation 1.5 of \cite{gamdiff})
\begin{align}\label{newref02}
h_X(x)=\frac{\alpha^{p} \beta^{q}}{\Gamma p \Gamma q} \begin{cases}
e^{\beta x}\int_{x}^{\infty}y^{p-1}(y-x)^{q-1}e^{-(\alpha+\beta)y}dy  ,&~x\in(0,\infty)\\
e^{-\alpha x}\int_{-x}^{\infty}y^{q-1}(y+x)^{p-1}e^{-(\alpha+\beta)y}dy  ,&~x\in(-\infty,0).
\end{cases}
\end{align}

\noindent
Some notable special cases and limiting distributions of the BG family are as follows
(see \cite{bilateral0}). Let $\overset{d}{=}$ denote the equality in distribution.
\begin{enumerate}
	\item[(i)] When $q=p$, then $BG(\alpha,p,\beta,p)\overset{d}{=}VG(\alpha,\beta,p)$ (variance-gamma  distribution).
	
	\item [(ii)] When $\beta=\alpha$ and $q=p$, then $BG(\alpha,p,\alpha,p)\overset{d}{=}SVG(\alpha,p)$ (symmetric variance-gamma distribution).
		
	\item [(iii)] The limiting case as $p\to \infty$,  the $SVG( \sqrt{2p}/\alpha,p)$ converges in law to a normal $\mathcal{N}(0,\alpha^2)$ distribution.
	
	\item [(iv)] The limiting case as $\beta \to \infty$, the $BG(\alpha,p,\beta,q)$ converges in law to gamma $Ga(\alpha,p)$ distribution.
\end{enumerate}
\noindent
Note that the $BG$ distributions are infinitely divisible and self-decomposable, see \cite{bilateral0}. Let now $X\sim BG(\alpha,p,\beta,q)$. Then the $k$-th cumulant (see \cite[Section 2]{bilateral0}) is given by
\begin{align}
C_{k}(X)&=(k-1)!\left(\frac{p}{\alpha^k}+(-1)^k \frac{q}{\beta^k}  \right),~k\geq 1. \label{cuforbg}
\end{align}
\noindent
In particular, the first two cumulants (that is, mean and variance) of $X$ are 
$C_1(X)=\mathbb{E}(X)= \left(\frac{p}{\alpha}-\frac{q}{\beta}\right) \text{ and } C_{2}(X)=\text{Var}(X)=\left(\frac{p}{\alpha^{2}}+\frac{q}{\beta^{2}}\right).$ For more details of BG distributions, we refer the reader to  \cite{bilateral0} and \cite{bilateral1}.

\subsection{Function spaces and probability metrics}\label{PP2:FS}
Let $\mathbb{N} = \{1,2,\ldots\} \text{ and }
\mathbb{N}_{0} = \mathbb{N} \cup \{0\}$ denote henceforth the set of positive and 
non-negative integers, respectively. 

The Schwartz space is defined as
\[
\mathcal{S}(\mathbb{R})
  := \Big\{ f \in C^{\infty}(\mathbb{R}) :
        \lim_{|x|\to\infty} |x^{m} f^{(n)}(x)| = 0,
        \ \ \forall\, m,n \in \mathbb{N}_{0}
     \Big\},
\]
where $f^{(n)}$ denotes the $n$-th derivative of $f$, with the convention $f^{(0)} = f$, and 
$C^\infty(\mathbb{R})$ is the space of smooth functions on $\mathbb{R}$. It is well known that the Forier-transform (FT) is an automorphism on $\mathcal{S}(\mathbb{R})$.  
For $f \in \mathcal{S}(\mathbb{R})$, we define its FT by $\widehat{f}(u)
   := \int_{\mathbb{R}} e^{-iux} f(x)\, dx,
    u \in \mathbb{R},$ and its inverse FT by $f(x)
   := \frac{1}{2\pi} \int_{\mathbb{R}} e^{iux} \widehat{f}(u)\, du,
   x \in \mathbb{R},$ see Stein and Shakarchi \cite{stein}.

\noindent
We now introduce the notion of smooth Wasserstein distance (see \cite{k0} or \cite{MS001}). For a fixed integer $r \in \mathbb{N}$, consider the function class
\begin{equation}\label{fs1}
\mathcal{W}_r := 
\Big\{ h:\mathbb{R}\to \mathbb{R} \;\Big|\; 
    h \text{ is $r$-times differentiable and } 
    \|h^{(k)}\| \leq 1,\; k=0,1,\ldots,r 
\Big\},
\end{equation}
where $\|h\| := \sup_{x \in \mathbb{R}} |h(x)|$. Then, for two random variables $Y$ and $Z$, the smooth Wasserstein distance of order $r$ is defined by
\begin{equation}\label{smwdis}
  d_r(Y,Z):=d_{\mathcal{W}_r}(Y,Z)
   := \sup_{h \in \mathcal{W}_r} 
      \left|\mathbb{E}[h(Y)] - \mathbb{E}[h(Z)]\right|.  
\end{equation}

Special choices of the function class $\mathcal{W}_r$ lead to well-known probability metrics:
\begin{align*}
 d_K(Y,Z) 
   &= \Big\{h:\mathbb{R}\to \mathbb{R} \;\big|\; 
          h=\mathbf{1}_{(-\infty, x]},~ x \in \mathbb{R}\Big\}, \\[0.3em]
 d_W(Y,Z)  &= \Big\{h:\mathbb{R}\to \mathbb{R} \;\big|\; 
          h \text{ is 1-Lipschitz, i.e., } |h(x)-h(y)| \leq |x-y| \Big\},\\[0.3em]
      d_{TV}(Y,Z)&=\Big\{h:\mathbb{R}\to \mathbb{R} \;\big|\; 
          h=\mathbf{1}_{A},~ A \in \mathfrak{B} (\mathbb{R})\Big\},
\end{align*}
where $d_K(Y,Z)$ is the Kolmogorov distance,  
$d_W(Y,Z)$ is the Wasserstein distance, and $d_{TV}(Y,Z)$ is  the total variation distance (see \cite{k0}). Since $\mathfrak{B} (\mathbb{R})$ are the Borel subsets of $\mathbb{R}$, we note that (see, \cite[Lecture 2]{Chatterjee2007})
\begin{align}\label{reltvk}
d_K(Y,Z) \leq d_{TV}(Y,Z).
\end{align}
\noindent
It is shown in Appendix~A of \cite{k0} that, for $r \geq 2$,
\begin{align}\label{ineq1}
d_{r-1}(Y,Z) &\leq 3\sqrt{2}\, \sqrt{\,d_r(Y,Z)}, \text{ and}\\
d_{1}(Y,Z) &\leq 
    \Big(3\sqrt{2}\Big)^{\sum_{i=1}^{r-1} 2^{-(i-1)}} 
    \Big(d_r(Y,Z)\Big)^{2^{-(r-1)}}.
    \label{mre1}
\end{align}
\noindent
Moreover, there exists a universal constant $c>0$ such that
\[
d_{K}(Y,Z) \;\leq\; c\, \sqrt{\,d_{1}(Y,Z)}.
\]
Combining this with \eqref{mre1}, we obtain the useful bound
\[
d_{K}(Y,Z) 
   \;\leq\; 
   c \Big(3\sqrt{2}\Big)^{\,1-2^{1-r}}
   \Big(d_r(Y,Z)\Big)^{2^{-r}}, 
   \qquad r \geq 2.
\]

\vfill
\noindent
Finally, the following order relationship between the distances holds for all $r \geq 1$ (see \cite[equation 2.16]{k0}):
\begin{equation}\label{metrico1}
d_r(Y,Z) \;\leq\; d_1(Y,Z) \;\leq\; d_W(Y,Z).
\end{equation}
\subsection{Essence of Stein's Method}
Stein's method is based on the observation that a real-valued rv $Z$ has distribution $F_Z$ if and only if there exists a suitable operator $\mathcal{A}$, called the Stein operator, such that
$$
\mathbb{E}\big[\mathcal{A}f(Z)\big] = 0,
$$
for all $f$ belonging to a ``nice'' class of functions $\mathcal{F}$. This characterization naturally leads to the Stein equation
\begin{equation}\label{normal2}
\mathcal{A}f(x) = h(x) - \mathbb{E}[h(Z)],
\end{equation}
where $h$ is a real-valued test function. Replacing $x$ by $Y$ and taking expectations on both sides of \eqref{normal2} yields
\begin{equation}\label{normal3}
\mathbb{E}[h(Y)] - \mathbb{E}[h(Z)] = \mathbb{E}\big[\mathcal{A}f(Y)\big].
\end{equation}
The identity \eqref{normal3} plays a crucial role in Stein's method. For a given test function $h$, bounding the quantity $\big|\mathbb{E}[h(Y)] - \mathbb{E}[h(Z)]\big|$ reduces to obtaining suitable estimates on the solution $f$ of the Stein equation \eqref{normal2} together with information about the distribution of $Y$. For further details and applications of Stein's method, we refer the reader to the monograph \cite{nourdin} and the references therein.

\vspace{-.25cm}

\section{Linear combinations of BG random variables}\label{probgdistribution}
\noindent
 In this section, we discuss various properties of the distribution of linear combinations of BG r.v.s. First, we obtain the exact distribution of difference of linear combinations of gamma r.v.s with arbitrary parameters, which is a generalization of Theorem 3.1 of \cite{pvns2009}. In \cite{DP1990}, the authors mentioned that the distribution of sum of independent gamma r.v.s is not expressible in a closed form and so discussed approximations and studied tail probabilities. Using the connection between the negative binomial and gamma distributions (see \cite{EZ1980} or \cite{VS2010}), Vellaisamy and Upadhye \cite{pvns2009} obtain the exact distribution of weighted sums of independent gamma random variables with arbitrary parameters. 
 \vspace{-.5cm}
\subsection{Notations}\label{nota} Let $X_1,X_2,\ldots,X_n$ be independent random variables, where $X_j \sim Ga(\alpha_j,p_j)$ and $Y_1,Y_2,\ldots,Y_n$ be independent random variables, where $Y_j \sim Ga(\beta_j,q_j)$. Let $p=\sum_{j=1}^{n}p_j$ and $q=\sum_{j=1}^{n}q_j$. For $ 1 \le j \le n$, let $w_j^{(1)},w_j^{(2)}>0$ be positive constants.

\noindent
Define now
\begin{align}\label{defTn}
	T_n= \sum_{j=1}^{n}(w_j^{(1)}X_j-w_j^{(2)}Y_j),~n\geq 1.
\end{align} 
\noindent Note when $w_j^{(1)}=w_j^{(2)},  1 \le j \le n,$ we get a linear combination of BG r.v.s.

Next, we introduce the following notations:
\begin{align}\label{ref1}
\alpha=\max_{1\leq j \leq n} \frac{\alpha_j}{w_j^{(1)}+\alpha_j} \text{ and } \beta=\max_{1\leq j \leq n} \frac{\beta_j}{w_j^{(2)}+\beta_j},
\end{align}
and for $i, n \in \mathbb{N}$,
\begin{align*}
	c_n=\prod_{j=1}^{n}\Big(  \frac{(1-\alpha)\alpha_j}{w_j^{(1)}\alpha}  \Big)^{p_j}, \quad  d_n=\prod_{j=1}^{n}\Big(  \frac{(1-\beta)\beta_j}{w_j^{(2)}\beta}  \Big)^{q_j} \\
	a_i=\frac{1}{i}\sum_{j=1}^{n}p_j\Big(\frac{1-(1-\alpha)\alpha_j}{w_j^{(1)}\alpha} \Big)^i, \quad   b_i=\frac{1}{i}\sum_{j=1}^{n}q_j\Big(\frac{1-(1-\beta)\beta_j}{w_j^{(2)}\beta} \Big)^i. 
\end{align*}

 Let $L_n$ and $M_n$ be two discrete r.v.s with probability distributions (see \cite{pvns2009}), for 
\begin{align}\label{ref2}
P(L_n=k)=c_n\gamma_k,~~  P(M_n=k)=d_n\delta_k, ~~k \ge 1,
\end{align}
 where $\gamma_0=1$, $\gamma_k=\frac{1}{k}\sum_{i=1}^{k}ia_i\gamma_{k-i}$ and $\delta_0=1$, $\delta_k=\frac{1}{k}\sum_{i=1}^{k}ib_i\delta_{k-i}$, for $k \ge 1.$.

\begin{thm}\label{THM1}
	Let $T_n$ be defined as in \eqref{defTn}. Then $T_n\sim BG(\frac{\alpha}{1-\alpha}, L_n+p,\frac{\beta}{1-\beta},M_n+q)$, where $\alpha, \beta$ are defined in \eqref{ref1} and the r.v.s $L_n$ and $M_n$ have distributions defined in \eqref{ref2}.
\end{thm}
\begin{proof}
	First note that 
	\begin{align}\label{gdd0}
		T_n	=\sum_{j=1}^{n}(w_j^{(1)}X_j-w_j^{(2)}Y_j)=\sum_{j=1}^{n}w_j^{(1)}X_j-\sum_{j=1}^{n}w_j^{(2)}Y_j.
	\end{align}
\vspace{-0.5cm}	

	\noindent
	 From Theorem 3.1 of \cite{pvns2009}, we get $\sum_{j=1}^{n}w_j^{(1)}X_j \sim Ga(\frac{\alpha}{1-\alpha},L_n+p)$ and $\sum_{j=1}^{n}w_j^{(2)}Y_j \sim Ga(\frac{\beta}{1-\beta}, M_n+q)$. It follows, from the definition of BG distribution,  that $T_n \sim BG (\frac{\alpha}{1-\alpha},L_n+p,\frac{\beta}{1-\beta}, M_n+q)$, which proves the result. 
\end{proof}
\noindent
Since the parameters $L_n$ and $M_n$ are random, the distribution of $T_n$ is indeed a mixture of BG distributions.

\noindent
The following corollary easily follows.
\vspace{-0.5cm}
\begin{cor}\label{LCgamma}
Let the conditions of Theorem \ref{THM1} hold and $\alpha,\beta$ be defined in \eqref{ref1}. Let $L_n$ and $M_n$ be defined in \eqref{ref2}.
\begin{itemize}
\item[(i)] If $\beta \to 1$, as $n\to \infty$, then $T_n \overset{\mathfrak{L}}{\to} X,$ where $X\sim Ga(\frac{\alpha}{1-\alpha},L_n+p).$

\item[(ii)] If $\alpha \to 1$, as $n\to \infty$, then $T_n \overset{\mathfrak{L}}{\to} -Y,$ where $Y\sim Ga(\frac{\beta}{1-\beta},M_n+q).$
\end{itemize}
\end{cor}
\begin{proof}
$(i)$ Using Theorem \ref{THM1} together with \eqref{bicf0} the cf of $T_n \sim BG (\frac{\alpha}{1-\alpha},L_n+p,\frac{\beta}{1-\beta}, M_n+q)$ can be seen as
\begin{align}\label{bircf0}
\phi_{T_n}(z)=\sum_{j=0}^{\infty}\sum_{k=0}^{\infty}\frac{P(L_n=j)P(M_n=k)}{\left(1 -
\frac{iz(1-\alpha)}{\alpha}\right)^{p+j} \left(1 +\frac{iz(1-\beta)}{\beta}\right)^{q+k}}, ~~z\in\mathbb{R}.
\end{align}
\noindent
Next taking $\beta \to 1$
in \eqref{bircf0}, we get
\begin{align}\label{bircf1}
\phi_X(z)=\sum_{j=0}^{\infty}\frac{P(L_n=j)}{\left(1 -\frac{iz(1-\alpha)}{\alpha}\right)^{p+j} }, ~~z\in\mathbb{R}.
\end{align}
 which is the cf of $X=\sum_{j=1}^{n}w_j^{(1)}X_j$, where $X\sim Ga(\frac{\alpha}{1-\alpha},L_n+p)$ (see Theorem 3.1 of \cite{pvns2009}).

\noindent
$(ii)$ Similarly, taking $\alpha \to 1$ in \eqref{bircf0}, Part (ii) follows.  This proves the result.
\end{proof}

\begin{rem}\label{rmkTn}
 When $w_j^{(1)}=w_j^{(2)}=w_j$, $j\geq 1$, Theorem \ref{THM1} yields the convolution of $n$-independent BG r.v.s with arbitrary parameters. Note Theorem \ref{THM1} gives a simple random parameter representation for the distribution of $T_n$, which may be helpful for analytical or inferential purposes. 
 \end{rem}
 \vspace{-0.5cm}
\noindent
Next, we obtain pdf, moments, cumulants and Stein characterizations of linear combinations of BG r.v.s. 
The Stein characterization will be required when we will consider approximation for the distribution of linear combination of BG r.v.s.

\subsection{Probability density function}\label{pdfTn}
Taking FT on \eqref{bircf0}, the pdf of $T_n\sim BG (\eta,L_n+p,\xi, M_n+q)$ has the integral form
\begin{align}\label{birpdf0}
h_{T_n}(x)&=\frac{1}{2\pi}\sum_{j=0}^{\infty}\sum_{k=0}^{\infty}\int_{\mathbb{R}}\frac{P(L_n=j)P(M_n=k) e^{-ixz}}{(1 -iz/\eta)^{p+j} (1 +iz/\xi)^{q+k}}~dz.
\end{align}
\noindent
where $\eta=\frac{\alpha}{1-\alpha}$ and $\xi=\frac{\beta}{1-\beta}$.

 Let
 \begin{align*}
  k_1(x):= & \int_{x}^{\infty}y^{p+j-1}(y-x)^{q+k-1}e^{-(\eta+\xi)y}dy  ,~x\in(0,\infty)\\
   k_2(x):= &\int_{-x}^{\infty}y^{q+k-1}(y+x)^{p+j-1}e^{-(\eta+\xi)y}dy  ,~x\in(-\infty,0). 
\end{align*}

Using Theorem \ref{THM1} together with equation \eqref{newref02}, an alternative form of the pdf of $T_n \sim BG (\eta,L_n+p,\xi, M_n+q)$ follows due to convolution structure:
\begin{align}\label{denlcbg0}
\nonumber	h_{T_{n}}(x)&=\sum_{j=0}^{\infty}\sum_{k=0}^{\infty}\frac{P(L_n=j)\eta^{p+j} P(M_n=k)\xi^{q+k}}{\Gamma (p+j) \Gamma (q+k)} \big( e^{\xi x}k_1(x)\mathbf{1}_{(0,\infty)}(x)\\
\nonumber&\quad\quad\quad\quad+  e^{-\eta x}k_2(x)\mathbf{1}_{(-\infty,0)}(x)\big)\\
\nonumber	&=c_nd_n\sum_{j=0}^{\infty}\gamma_j\sum_{k=0}^{\infty}\delta_k\frac{\eta^{p+j} \xi^{q+k}}{\Gamma (p+j) \Gamma (q+k)} \big( e^{\xi x}k_1(x)\mathbf{1}_{(0,\infty)}(x)\\
&\quad\quad\quad\quad+  e^{-\eta x}k_2(x)\mathbf{1}_{(-\infty,0)}(x)\big)
\end{align}
\noindent
Let $ F(a,b,x)$ denote the confluent hypergeometric function (see \cite[equation (78)]{cam2001}), given by 
\begin{align}\label{chf1}
 F(a,b,x)=\frac{1}{\Gamma a}\displaystyle\int_{0}^{\infty}e^{-xt}t^{a-1}(1+t)^{b-a-1}dt, ~\text{Re}(a),\text{Re}(b),\text{Re}(x)>0.
\end{align}
\noindent
From \eqref{denlcbg0}, consider the first integral
\begin{align}\label{chf2}
	\nonumber I_1(x)&=\int_{x}^{\infty}y^{p+j-1}(y-x)^{q+k-1}e^{-(\eta+\xi)y}dy  ,~x\in(0,\infty)\\
	\nonumber &=e^{-(\eta+\xi)x}\int_{0}^{\infty}(x+u)^{p+j-1}u^{q+k-1}e^{-(\eta+\xi)u}du\\
	&=e^{-(\eta+\xi)x}x^{p+q+j+k-1}\displaystyle\int_{0}^{\infty}e^{-(\eta+\xi)xt}t^{q+k-1}(1+t)^{p+j-1}dt .
\end{align}
\noindent
Using \eqref{chf1} in \eqref{chf2}, we write
\begin{align}\label{chf3}
I_1(x)=e^{-(\eta+\xi)x}x^{p+q+j+k-1}\Gamma(q+k)F(q+k,p+q+j+k, (\eta+\xi)x), ~x\in (0,\infty).
\end{align}
\noindent
From \eqref{denlcbg0}, consider the second integral
\begin{align}\label{chf4}
I_2(x)=\int_{-x}^{\infty}y^{q+k-1}(y+x)^{p+j-1}e^{-(\eta+\xi)y}dy  ,~x\in(-\infty,0).
\end{align}
\noindent
Following the steps similar to previous integral, we have for $~x\in (-\infty,0)$,
\begin{align}\label{chf5}
I_2(x)=e^{(\eta+\xi)x}(-x)^{p+q+j+k-1}\Gamma(q+k)F(p+j,p+q+j+k, -(\eta+\xi)x). 
\end{align}
\noindent
Using \eqref{chf3} and \eqref{chf5} in \eqref{denlcbg0}, we can write the pdf of $T_n$ as
\begin{align}
 h_{T_{n}}(x)&=c_nd_n\sum_{j=0}^{\infty}\gamma_j\sum_{k=0}^{\infty}\delta_k\frac{\eta^{p+j} \xi^{q+k}}{\Gamma (p+j) \Gamma (q+k)} K_2(x),
\end{align}
where
\begin{align}
	K_2(x)= \small{\begin{cases}
			e^{-\eta x}x^{p+q+j+k-1}\Gamma(q+k)F(q+k,p+q+j+k, (\eta+\xi)x),  ~~x\in(0,\infty)\\
			e^{\xi x}(-x)^{p+q+j+k-1}\Gamma(q+k)F(p+j,p+q+j+k, -(\eta+\xi)x),  ~~x\in(-\infty,0).
	\end{cases}}
\end{align}

\vspace{-.5cm}
\begin{rem}
(i) Scheuer \cite{sch1988} obtained the pdf of $S_n=\sum_{i=1}^{n}X_i,$ where $X_i \sim Ga (\alpha_i,p_i)$ are independent r.v.s, in terms of double series involving higher order derivatives in the inner terms. Later, Vellaisamy and Upadhye \cite[Equation 4.4]{pvns2009} obtained a  simpler expression for the pdf of $S_n$, represented as a single series involving the distribution of the random variable (r.v.) $L_n$. 

\item[(ii)] Note that, if $\xi \to \infty$ in \eqref{birpdf0}, we get 
\begin{align}\label{denlcg}
\nonumber h_{S_n}(x)&=\frac{1}{2\pi}\sum_{j=0}^{\infty}\int_{\mathbb{R}}\frac{P(L_n=j) e^{-ixz}}{(1 -iz/\eta)^{p+j} }~dz\\
\nonumber &=\frac{1}{2\pi}\sum_{j=0}^{\infty}P(L_n=j)\eta^{p+j}\int_{\mathbb{R}}(\eta-iz)^{-(p+j)} e^{-ixz}~dz\\
\nonumber &=\sum_{j=0}^{\infty}P(L_n=j)\frac{\eta^{p+j} x^{p+j-1}e^{-\eta x} }{\Gamma (p+j)}\\
&=\sum_{j=0}^{\infty}c_n\gamma_j\frac{\eta^{p+j} x^{p+j-1}e^{-\eta x} }{\Gamma (p+j)},~x>0,
\end{align}
\noindent
which is the pdf of $S_n$, see \cite[Equation 4.4]{pvns2009}.
\item[(iii)] Similarly, if $\eta \to \infty$ in \eqref{birpdf0}, we get 
\begin{align*}
h_{S_n}(x)
&=\frac{1}{2\pi}\sum_{k=0}^{\infty}P(M_n=k)\xi^{p+j}\int_{\mathbb{R}}(\xi+iz)^{-(q+j)} e^{-ixz}~dz\\
&=\sum_{j=0}^{\infty}P(M_n=j)\frac{\xi^{q+j} (-x)^{q+j-1}e^{\xi x} }{\Gamma (q+j)}\\
&=\sum_{j=0}^{\infty}d_n\delta_j\frac{\eta^{q+j} (-x)^{q+j-1}e^{\xi x} }{\Gamma (q+j)},~x<0,
\end{align*}
\noindent
which is the pdf of $V_n=-\sum_{j=1}^{n}w_j^{(2)}Y_j$.

\end{rem}
\vspace{-.4cm}
\subsection{Moment generating function}\label{cflcbg} By Theorem \ref{THM1}, the mgf of $T_n$ is given by
\begin{align}
\label{mggd1}	\mathbb{E}(e^{zT_n}) &=\bigg(\sum_{j=0}^{\infty}P(L_n=j) \left( \frac{\eta}{\eta -z} \right)^{p+j}\bigg)\bigg(\sum_{k=0}^{\infty}P(M_n=k)\left( \frac{\xi}{\xi +z} \right)^{q+k}\bigg),~z\in\mathbb{R}\\
&=\mathbb{E}_{L_n}\bigg(\left( \frac{\eta}{\eta -z} \right)^{p+L_n}\bigg) \mathbb{E}_{M_n}\bigg(\left( \frac{\xi}{\xi +z} \right)^{q+M_n}\bigg),~z\in\mathbb{R},\label{ggd1}
\end{align}
\noindent
where $\mathbb{E}_{L_n}$ and $\mathbb{E}_{M_n}$ are expectations with respect to r.v.s $L_n$ and $M_n$, respectively. 
\vspace{-.2cm}
\subsection{Moments}
 Note from Theorem \ref{THM1} that $T_n:=S_n-V_n$, where  $S_n \sim Ga(\eta,L_n+p)$ and $V_n \sim Ga(\xi, M_n+q)$ (see \eqref{gdd0}). Hence, we have for $k\in \mathbb{N},$
\begin{align}\label{momentformula}
\nonumber	\mathbb{E}\left( T_n^k  \right)&=\mathbb{E}\left( (S_n-V_n)^k  \right)\\
\nonumber	&=\sum_{j=0}^{k}\binom{k}{j}(-1)^{j}\mathbb{E}(S_n^{k-j})\mathbb{E}(V_n^j)\\
\nonumber	&=\sum_{j=0}^{k}\binom{k}{j}(-1)^{j}\frac{1}{\eta^{k-j} \xi^j}\bigg(\sum_{l=0}^{\infty}P(L_n=l) \frac{\Gamma (l+p+k-j)}{\Gamma(l+p)}   \bigg)\\
\nonumber	&\quad\quad\quad\times\bigg(\sum_{m=0}^{\infty}P(M_n=m) \frac{\Gamma (m+q+j)}{\Gamma (m+q)}   \bigg)\\
\nonumber	&=\sum_{j=0}^{k}\binom{k}{j}(-1)^{j}\frac{1}{\eta^{k-j} \xi^j}\bigg(\sum_{l=0}^{\infty}c_n\gamma_l \frac{\Gamma (l+p+k-j)}{\Gamma(l+p)}   \bigg)\\
	&\quad\quad\quad\quad\quad\times\bigg(\sum_{m=0}^{\infty}d_n\delta_m \frac{\Gamma (m+q+j)}{\Gamma (m+q)}   \bigg).
\end{align}
\subsection{Characterizations}
 In this section, we derive a Stein characterization for $T_n$. Before stating the result, recall that the BG distributions are infinitely divisible and self-decomposable, see \cite{bilateral0}. Also, $T_n$ can be viewed as $n$-independent sums of BG r.v.s and it has BG distribution, but with random parameters. Hence, the distribution of $T_n$ is infinitely divisible and self-decomposable. Moreover, the L\'evy-Khintchine representation of the cf of $T_n$ as \eqref{e1} is given by
\begin{align}
\nonumber\phi_{T_n}(z)&=\mathbb{E}\bigg(e^{iz T_n} \bigg)=\mathbb{E}\bigg(e^{iz \left(\sum_{j=1}^{n} (w_j^{(1)}X_j-w_j^{(2)}Y_j\right) }\bigg),~z\in\mathbb{R}\\
\label{01cflcbg}&=\prod_{j=1}^{n}\left(\frac{\alpha_j}{\alpha_j-izw_{j}^{(1)}} \right)^{p_j}  \left(\frac{\beta_j}{\beta_j+izw_{j}^{(2)}}\right)^{q_j} ,~z\in\mathbb{R} \\
\label{lccf1}&=\prod_{j=1}^{n}\exp \left(\displaystyle\int_{\mathbb{R}}(e^{izu}-1) \nu_j(du)  \right),~z\in\mathbb{R} ~~(\text{use Frullani identity}),
\end{align} 
where $\nu_j(du)=\left(\frac{p_j}{u}e^{-\frac{\alpha_j u}{w_{j}^{(1)}}}\mathbf{1}_{(0,\infty)}(u)-\frac{q_j}{u}e^{-\frac{\beta_j |u|}{w_{j}^{(2)}}}\mathbf{1}_{(-\infty,0)}(u)\right)du.$ Also, the cf \eqref{lccf1} is
\begin{align}\label{lcbg2}
\phi_{T_{n}}(z)=\exp\left(\displaystyle\int_{\mathbb{R}}(e^{izu}-1) \nu_{T_n}(du)  \right), 	~z\in\mathbb{R},
\end{align}
\noindent
where $\nu_{T_n}$ is the L\'evy measure given by
\begin{align}\label{lcbg1}
\nu_{T_n}(du)=\frac{1}{u}\sum_{j=1}^{n}\left(p_je^{-\frac{\alpha_j u}{w_{j}^{(1)}}}\mathbf{1}_{(0,\infty)}(u)-q_je^{-\frac{\beta_j |u|}{w_{j}^{(2)}}}\mathbf{1}_{(-\infty,0)}(u) \right)du.
\end{align}
\noindent
When $n=1$, $w_{j}^{(1)}=w_{j}^{(2)}=1$, $\nu_{T_1}$ is the L\'evy measure of $BG(\alpha_1,p_1,\beta_1,q_1)$ distribution, see \cite{bilateral0}. Also, the $k$th order cumulant $C_k(T_n):=i^{-k}\frac{d^k}{dz^k}\log	\phi_{T_n}(z)|_{z=0}$ is
\begin{align}\label{cumu}
	C_k(T_n)=(k-1)! \int_{\mathbb{{R}}}u^{k}\nu_{T_n}(du),~k\in\mathbb{N}.
\end{align} 
\begin{pro}\label{th1}
	 Let $T_n$ be defined as in \eqref{defTn}. Then,
	\begin{align}\label{PP2:StenIdTn}
	\mathbb{E}\bigg(T_nf(T_n)-\displaystyle\int_{\mathbb{{R}}}uf(T_n+u)\nu_{T_n}(du) \bigg)=0,~~f\in\mathcal{S}(\mathbb{R}),
	\end{align}
	where $\nu_{T_n}$ is the L\'evy measure given in \eqref{lcbg1}.
\end{pro}

\vspace{-.6cm}
\begin{proof}
	Taking logarithms on both sides of (\ref{lcbg2}), and differentiating with respect to $z$,
	\begin{equation}\label{PP2:e4}
	\phi^{\prime}_{T_n}(z)=i\phi_{T_n}(z)\int_{\mathbb{R}}ue^{izu}\nu_{T_n}(du).
	\end{equation}
	
	\noindent
	Let $h_{T_n}(x)$ denote density of $T_n$. Then
	\begin{equation}\label{PP2:e5}
	\phi_{T_n}(z)=\displaystyle\int_{\mathbb{R}}e^{izx}h_{T_n}(x)dx~~\text{and}~~\phi^{\prime}_{T_n}(z)=i\displaystyle\int_{\mathbb{R}}xe^{izx}h_{T_n}(x)dx.
	\end{equation}
	
	\noindent
	Substituting \eqref{PP2:e5} into  \eqref{PP2:e4} and rearranging the integrals, we have
	\begin{align}	
	\displaystyle\int_{\mathbb{R}}xe^{izx}h_{T_n}(x)dx-\phi_{T_n}(z)\int_{\mathbb{R}}ue^{izu}\nu_{T_n}(du)=0\label{PP2:e6}
	\end{align}
	\noindent
	The second integral of \eqref{PP2:e6} can be written as
	\begin{align}
	\nonumber \left(\int_{\mathbb R}ue^{izu}\nu_{T_n}(du)\right)\phi_{T_n}(z)&=\int_{\mathbb R}\int_{\mathbb R}ue^{izu}e^{izx}h_{T_n}(x)dx\nu_{T_n}(du)\\
	\nonumber &=\int_{\mathbb R}\int_{\mathbb R}ue^{iz(u+x)}\nu_{T_n}(du)h_{T_n}(x)dx\\
	\nonumber &=\int_{\mathbb R}\int_{\mathbb R}ue^{izy}\nu_{T_n}(du)h_{T_n}(y-u)dy\\
	&=\int_{\mathbb R}e^{izx}\int_{\mathbb R}uh_{T_n}(x-u)\nu_{T_n}(du)dx.\label{PP2:e7}
	\end{align}
	\noindent
	Substituting \eqref{PP2:e7} into \eqref{PP2:e6}, we have
	\begin{align}
	\nonumber	0&=\displaystyle\int_{\mathbb{R}}xe^{izx}h_{T_n}(x)dx-\int_{\mathbb R}e^{izx}\int_{\mathbb R}uh_{T_n}(x-u)\nu_{T_n}(du)dx\\
	&=\displaystyle\int_{\mathbb{R}}e^{izx} \left( xh_{T_n}(x)-\int_{\mathbb R}uh_{T_n}(x-u)\nu_{T_n}(du) \right)dx. \label{PP2:e8}
	\end{align}
	\noindent
	Applying FT to \eqref{PP2:e8}, multiplying with $f\in \mathcal{S}(\mathbb{R}),$ and integrating over $\mathbb{R},$ we get
		\begin{align}\label{PP2:e9}
	\displaystyle\int_{\mathbb{R}}f(x) \left( xh_{T_n}(x)-\int_{\mathbb R}uh_{T_n}(x-u)\nu_{T_n}(du) \right)dx=0.
	\end{align}
	\noindent
	The second integral of \eqref{PP2:e9} can be seen as
	\begin{align}
	\nonumber \int_{\mathbb R}\int_{\mathbb R}uf(x)h_{T_n}(x-u)\nu_{T_n}(du)dx&=\int_{\mathbb R}\int_{\mathbb R}uf(y+u)h_{T_n}(y)\nu_{T_n}(du)dy\\ 
	&=\mathbb{E}\left(\int_{\mathbb R}f(T_n+u)u\nu_{T_n}(du)\right).\label{PP2:e10}
	\end{align}
	\noindent
	Substituting \eqref{PP2:e10} into \eqref{PP2:e9}, we have
	\begin{align*}
	\mathbb{E}\left(T_nf(T_n)-\displaystyle\int_{\mathbb{R}}f(T_n+u)u\nu_{T_n}(du) \right)=0,
	\end{align*}
	\noindent
	which proves \eqref{PP2:StenIdTn}.
	
	\noindent
	Assume conversely, \eqref{PP2:StenIdTn} holds for $\nu_{T_n}$ defined in \eqref{lcbg1}. For any $s \in \mathbb{R}$, let $f(x)=e^{isx}$, $x\in \mathbb{R}$, then (\ref{PP2:StenIdTn}) becomes
	\begin{align*}
	\mathbb{E}\left(T_ne^{isT_n}\right) &= \mathbb{E}\left(\int_{\mathbb R}e^{is(T_n+u)}u\nu_{T_n}(du)\right)\\
	&=\mathbb{E}\left(e^{isT_n}\int_{\mathbb R}e^{isu}u\nu_{T_n}(du)\right).
	\end{align*}
	\noindent
	Setting $\phi_{T_n}(s)=\mathbb{E}(e^{isX})$, then
	\begin{equation}\label{suff1}
	\phi_{T_n}^{\prime}(s)=i\phi_{T_n}(s)\int_{\mathbb R}e^{isu}u\nu_{T_n}(du).
	\end{equation}
	\noindent
	Integrating out the real and imaginary parts of (\ref{suff1}) leads, for any $z\geq 0$, to
	\begin{align*}
	\phi_{T_n}(z)&=\exp \left(i\int_{0}^{z}\int_{\mathbb R}e^{isu}u\nu_{T_n}(du)ds  \right)\\
	&=\exp\left(i\int_{\mathbb R}\int_{0}^{z}e^{isu}dsu\nu_{T_n}(du)  \right)\\
	&=\exp\left(\int_{\mathbb R}(e^{izu}-1)\nu_{T_n}(du)  \right).
	\end{align*}
	
	\noindent
	A similar computation for $z\leq0$ completes the derivation of cf of $T_n$. 
\end{proof}

\vspace*{-.7cm}
\subsection{Stein equation}\label{SEBGD}
Let $T_n$ be defined as in \eqref{defTn}. Then, from Proposition \ref{th1}, a Stein equation for $T_n$, is
\begin{equation}\label{PP2:e15}
\mathcal{A}f(x)=-xf(x)+\displaystyle\int_{\mathbb{R}}f(x+u)u\nu_{T_n}(du)= h(x)-\mathbb{E}(h(T_n)),
\end{equation}
\noindent
where $h\in \mathcal{M}$, a class of test functions. For a given $h\in \mathcal{M}$, we now use the semigroup technique to obtain $f_h$ that satisfies the integral equation \eqref{PP2:e15}. Barbour \cite{k8} developed this method for solving Stein equations, and Arras and Houdr\'e \cite{k0} extended it for infinitely divisible distributions with a finite first moment. Following Barbour's approach \cite{k8}, we choose a family of operators $(P_{t})_{t\geq0}$, for all $x\in\mathbb{R}$, as
\begin{equation}\label{PP2:e16}
P_{t}f(x):=\frac{1}{2\pi}\int_{\mathbb{R}}\widehat{f}(y)e^{iyxe^{-t} }\phi_t(y)dy, ~~f\in\mathcal{S}(\mathbb{{R}}),
\end{equation}
\noindent 
where $\hat{f}$ is the FT of $f$ and
\begin{align}\label{nPP2:a15}
\phi_t(y):=\frac{\phi_{T_n}(y)}{\phi_{T_n}(e^{-t}y)},~y\in \mathbb{{R}}.
\end{align}
\noindent
Since the distribution of $T_n$ is self-decomposable,  $\phi_{t}$ is a cf \cite[p.90]{sato} of some rv $X_{(t)}$, say. That is, for all $y\in \mathbb{R},$ and $t\geq 0,$
\begin{align}\label{PP2:a15}
\phi_t(y)=\displaystyle\int_{\mathbb{R}}e^{iy u}F_{X_{(t)}}(du),
\end{align}
\noindent
where $F_{X_{(t)}}$ is the law of $X_{(t)}$. Using \eqref{PP2:a15}, we get
\begin{align}
\nonumber  P_tf(x)&=\frac{1}{2\pi}\int_{\mathbb R}\int_{\mathbb{R}}\widehat{f}(y)e^{iy xe^{-t}}e^{iy u}F_{X_{(t)}}(du)dy\\
\nonumber  &=\frac{1}{2\pi}\int_{\mathbb R}\int_{\mathbb{R}}\widehat{f}(y)e^{iy (u+xe^{-t})}F_{X_{(t)}}(du)dy\\
&=\displaystyle\int_{\mathbb{R}}f(u+xe^{-t})F_{X_{(t)}}(du),\label{PP2:a17}
\end{align}
\noindent
where the last step follows by inverse FT (see Section \ref{PP2:FS}).
\begin{pro}\label{PP2:proSem}
	Let $(P_{t})_{t\geq0}$ be a family of operators defined in \eqref{PP2:e16}. Then
	\begin{itemize}
		\item [(i)] $(P_{t})_{t\geq 0}$ is a $\mathbb{C}_0$-semigroup on $\mathcal{S}(\mathbb{{R}})$.
		\item [(ii)] Its generator $L$ is given by
		\begin{align}
		\nonumber Lf(x)&=-xf^{\prime}(x)+\displaystyle\int_{\mathbb R}f^{\prime}(x+u)u\nu_{T_n}(du),~~f\in\mathcal{S}(\mathbb{R})\\
		\label{e7}&=\mathcal{A}f^{\prime}(x),
		\end{align}	
		
		\noindent
		where $\mathcal{A}$ is defined in \eqref{PP2:e15}.
	\end{itemize}
\end{pro}
\noindent
Following the steps similar to Proposition 3.8 and Lemma 3.10 of \cite{k1}, the proof follows.

\noindent
Next, we provide a solution of the Stein equation in \eqref{PP2:e15}.

\vfill
\begin{thm}\label{thmsol}
	Let $T_n$ be defined as in \eqref{defTn} and $h\in\mathcal{W}_{r}$, defined in \eqref{fs1}. Then the function $f_{h}:\mathbb{R}\to \mathbb{R}$ defined by 
	\begin{equation}\label{PP2:SolSe}
	f_{h}(x):=-\displaystyle
	\int_{0}^{\infty}\frac{d}{dx}P_{t}h(x)dt,
	\end{equation}
	solves \eqref{PP2:e15}.	
\end{thm}

\vspace{-.3cm}
\begin{proof}
	Let
	 $$g_{h}(x)=-\displaystyle\int_{0}^{\infty}\bigg(P_{t}(h)(x)-\mathbb{E}h(T_n)  \bigg)dt .$$
	
	\noindent
	Then $g_{h}^{\prime}(x)=f_{h}(x).$ Also from \eqref{e7}, we get
	\begin{align}
	\nonumber\mathcal{A}f_{h}(x)&=-xf_{h}(x)+\int_{\mathbb{R}} f_{h}(x+u)u\nu_{T_n} (du)=Lg_{h}(x) \\
	\nonumber&=-\displaystyle\int_{0}^{\infty}LP_{t}(h)(x)dt\\
	\nonumber&=-\displaystyle\int_{0}^{\infty}\frac{d}{dt}P_{t}h(x)dt\text{ (see \cite[p.68]{nourdin})}\\
	\nonumber&=P_{0}h(x)-P_{\infty}h(x)\\
	\nonumber&=h(x)-\mathbb{E}h(T_n)~(\text{by Proposition \ref{PP2:proSem}}).
	\end{align}
	\noindent
	Hence, $f_{h}$ is the solution to \eqref{PP2:e15}. 	
\end{proof}
\vspace{-0.4cm}
\noindent
Next, we obtain some regularity conditions of $f_{h}$. 
\vspace{-0.4cm}
\begin{thm}\label{th3}
	For $h\in\mathcal{W}_{k+1}$, let $f_{h}$ be defined in \eqref{PP2:SolSe}. Then
	\begin{align}\label{PP2:pr1}
	\|f_{h}^{(k)}\|\leq \frac{1}{k+1} ,~k\geq 0,
	\end{align}	
	where $f^{(k)}$, $k\geq 1$, denotes the $k$th derivative of $f$ with $f^{(0)}=f$. Also,
	\begin{align}
		|xf_{h}^{\prime}(x)| \leq 2+\frac{1}{2}|\mathbb{E}(T_n)| \text{ and }|xf_{h}^{\prime\prime}(x)| \leq 2+\frac{1}{3}|\mathbb{E}(T_n)|.
	\end{align}
\end{thm}
\begin{proof}
	For $h\in \mathcal{W}_{k+1}$,
	\begin{align*}
	\|f_h\|&=\sup_{x\in\mathbb{R}} \left|-\displaystyle
	\int_{0}^{\infty}\frac{d}{dx}P_{t}h(x)dt\right|\\
	&=\sup_{x\in\mathbb{R}} \left| -\displaystyle
	\int_{0}^{\infty}e^{-t}\int_{\mathbb{R}}h^{(1)}(xe^{-t}+y)F_{X_{(t)}}(dy)dt \right|\\
	&\leq \|h^{(1)}\| \left|\int_{0}^{\infty}e^{-t}dt \right|
	= \|h^{(1)}\|\leq 1.
	\end{align*}
	\noindent
	Since $f_{h}$ is $k$-times differentiable, we have 
	\begin{align*}
	\|f_h^{(1)}\|&=\sup_{x\in\mathbb{R}} \left| -\displaystyle
	\int_{0}^{\infty}e^{-2t}\int_{\mathbb{R}}h^{(2)}(xe^{-t}+y)F_{X_{(t)}}(dy)dt \right|\\
	&\leq \|h^{(2)}\| \left|\int_{0}^{\infty}e^{-2t}dt \right|\\	
	&= \frac{\|h^{(2)}\|}{2} \leq \frac{1}{2}.
	\end{align*}
	\noindent
	 Indeed, it follows by induction that $\|f_{h}^{(k)}\|\leq \frac{1}{k+1},~k\geq 1.$

	\vskip 1ex
	\noindent
	Now differentiating both sides of \eqref{PP2:e15} gives
	\begin{align}\label{sed1}
		-xf^{\prime}(x)-f(x)+\int_{\mathbb{{R}}}f^{\prime}(x+u)u\nu_{T_n}(du)=h^{\prime}(x).
	\end{align}
	\noindent
	So,
	\begin{align*}
		|xf^{\prime}(x)|& \leq \|f\| +\|h^{\prime}\|+\left|\int_{\mathbb{{R}}}f^{\prime}(x+u)u\nu_{T_n}(du)\right|\\
		&\leq \left( 2 + \|f^\prime\|\left|\int_{\mathbb{{R}}}u\nu_{T_n}(du)\right| \right)\\
		& =2+\frac{1}{2}|C_1(T_n)| ~(\text{see } \eqref{cumu})\\
		&=2+\frac{1}{2}|\mathbb{E}(T_n)|.
	\end{align*}
	\noindent
	Again differentiating both sides of \eqref{sed1} gives
	\begin{align}
			-xf^{\prime\prime}(x)-2f^{\prime}(x)+\int_{\mathbb{{R}}}f^{\prime\prime}(x+u)u\nu_{T_n}(du)=h^{\prime\prime}(x).
	\end{align}
	\noindent
	So,
	\begin{align*}
		|xf^{\prime\prime}(x)| &\leq 2\|f^{\prime}\| +\|h^{\prime\prime}\|+ \|f^{\prime\prime}\|\left|\int_{\mathbb{{R}}}u\nu_{T_n}(du)\right|\\
		& \leq 2 +\frac{1}{3}|C_1(T_n)|\\
		&=2+\frac{1}{3}|\mathbb{E}(T_n)|.
	\end{align*}
	\noindent
	This proves the result.
\end{proof}
\section{Approximation of linear combinations of BG random variables}\label{BSBGD}
\noindent
In this section, we present bounds for the distributional approximations of linear combination of independent BG r.v.s to some probability distributions. We use Stein's method for probability approximations to derive our bounds.
\subsection{Approximation for sums}Our first result yields an error bound for the distributional approximation of two linear combinations of independent BG r.v.s. We refer the reader to \cite{Ransum2024}, \cite{roos2003} and \cite{PVBC} for a number of similar bounds for the total variation distance, between the distributions of two sums of non-negative integer-valued r.v.s. 
\begin{thm}\label{twosumsthm}
	For $w_j^{(1)},w_j^{(2)}, \pi_j^{(1)},\pi_j^{(2)}>0$ and $j\in \mathbb{N} \{ 0\}$, let $T_n^{(1)}=T_n= \sum_{j=1}^{n}(w_j^{(1)}X_j-w_j^{(2)}Y_j)$ and  $T_n^{(2)}= \sum_{j=1}^{n}(\pi_j^{(1)}X_j-\pi_j^{(2)}Y_j)$ be defined as in \eqref{defTn}. Then
	\begin{align}\label{twosumbd}
		d_W(T_n^{(1)},T_n^{(2)}) &\leq c_1 \sum_{j=1}^{n}\frac{p_j}{\sqrt{2\alpha_j}} \frac{ \big(w_j^{(1)} - \pi_{j}^{(1)}\big) }{\big(w_j^{(1)} + \pi_{j}^{(1)}\big)^{1/2}}+c_2\sum_{j=1}^{n}\frac{q_j}{\sqrt{2\beta_j}}\frac{\big(w_j^{(2)}-\pi_{j}^{(2)}\big)}{\big(w_j^{(2)}+\pi_{j}\big)^{1/2}},
	\end{align}
	where $c_1$ and $c_2$ are two positive constants.
\end{thm}
\vspace{-0.6cm}
 \begin{proof}
 	Let $(f,h)$ be the solution pair of \eqref{PP2:e15}. Then replacing $x$ by $T_n^{(1)} $ in the Stein equation and taking expectation, we get	
 	\begin{align}\label{ransum1}
 	\nonumber	\mathbb{E}[h(T_n^{(1)})]-\mathbb{E}[h(T_n^{(2)})]=&\mathbb{E}\left(-T_n^{(1)}f(T_n^{(1)})+\displaystyle\int_{\mathbb{R}}f(T_n^{(1)}+u)u\nu_{T_n^{(2)}}(du)  \right)\\
 \nonumber	=&\mathbb{E}\bigg[\left(-T_n^{(1)}f(T_n^{(1)})+\displaystyle\int_{\mathbb{R}}f(T_n^{(1)}+u)u\nu_{T_n^{(2)}}(du)  \right)\\
 	&\quad\quad-\left(-T_n^{(1)}f(T_n^{(1)})+\displaystyle\int_{\mathbb{R}}f(T_n^{(1)}+u)u\nu_{T_n^{(1)}}(du)  \right)\bigg],
 	\end{align}
 	\noindent
 	since $\mathbb{E}\left(-T_n^{(1)}f(T_n^{(1)})+\displaystyle\int_{\mathbb{R}}f(T_n^{(1)}+u)u\nu_{T_n^{(1)}}(du)  \right)=0, ~f\in\mathcal{S}(\mathbb{{R}}).$ So, for any $h\in \mathcal{M}_W$, \eqref{ransum1} can be written as
 	\begin{align}\label{2ransum}
 		\nonumber d_W(T_n^{(1)},T_n^{(2)}) &\leq \bigg|\mathbb{E}\displaystyle\int_{\mathbb{R}}f(T_n^{(1)}+u)u\nu_{T_n^{(2)}}(du)  -\displaystyle\int_{\mathbb{R}}f(T_n^{(1)}+u)u\nu_{T_n^{(1)}}(du)  \bigg|\\
 		\nonumber &=\bigg|\mathbb{E}\sum_{j=1}^{n}\bigg(p_j\displaystyle\int_{0}^{\infty}f(T_n^{(1)}+u)e^{-\alpha_j/\pi_{j}^{(1)} u}du\\
 		\nonumber &\quad - q_j\displaystyle\int_{0}^{\infty}f(T_n^{(1)}-u)e^{-\beta_j/\pi_{j}^{(2)} u}du\bigg)\\
 		\nonumber &\quad\quad -\sum_{j=1}^{n}\bigg(p_j\displaystyle\int_{0}^{\infty}f(T_n^{(1)}+u)e^{-\alpha_j/w_{j}^{(1)} u}du\\
 		\nonumber &\quad\quad\quad - q_j\displaystyle\int_{0}^{\infty}f(T_n^{(1)}-u)e^{-\beta_j/w_{j}^{(2)} u}du\bigg) \bigg|\\
 		\nonumber &\quad\quad\quad\quad\quad\quad\quad\quad\quad\quad\quad\quad(\text{split the measures } \nu_{T_n^{(1)}} \text{ and } \nu_{T_n^{(2)}})\\
 		\nonumber &=\bigg|\mathbb{E} \sum_{j=1}^{n}\bigg(p_j\displaystyle\int_{0}^{\infty}f(T_n^{(1)}+u)\left(e^{-\alpha_j/\pi_{j}^{(1)} u}-e^{-\alpha_j/w_{j}^{(1)} u}\right)du\\
 		\nonumber &\quad\quad\quad\quad\quad\quad-q_j\displaystyle\int_{0}^{\infty}f(T_n^{(1)}-u)\left(e^{-\beta_j/\pi_{j}^{(2)} u}-e^{-\beta_j/w_{j}^{(2)} u}\right)du\bigg) \bigg|\\
 		\nonumber &\leq\sum_{j=1}^{n}p_j \bigg|\bigg(\int_{0}^{\infty}\bigg(e^{-\alpha_j/\pi_{j}^{(1)} u}-e^{-\alpha_j/w_{j}^{(1)} u}\bigg)^2 du\bigg)^{1/2}\\
 		\nonumber&\quad\quad\quad \times
\bigg(\mathbb{E}\displaystyle\int_{0}^{\infty}f^{2}(T_n^{(1)}+u)du\bigg)^{1/2}\bigg|\\
\nonumber&\quad+\sum_{j=1}^{n}q_j \bigg|\bigg(\int_{0}^{\infty}\bigg(e^{-\beta_j/\pi_{j}^{(2)} u}-e^{-\beta_j/w_{j}^{(2)} u}\bigg)^2 du\bigg)^{1/2}\\
&\quad\quad\quad\quad\times
\bigg(\mathbb{E}\displaystyle\int_{0}^{\infty}f^{2}(T_n^{(1)}-u)du\bigg)^{1/2}\bigg|,
 \end{align}
 \noindent
 where the last inequality follows by the Cauchy-Schwartz inequality.
 
  Next, observe that, since $f\in\mathcal{S}(\mathbb{R})$, there exists two positive constants $c_1$ and $c_2$ such that $\bigg(\mathbb{E}\displaystyle\int_{0}^{\infty}f^{2}(T_n^{(1)}-u)du\bigg)^{1/2}<c_1$ and $\bigg(\mathbb{E}\displaystyle\int_{0}^{\infty}f^{2}(T_n^{(1)}-u)du\bigg)^{1/2}<c_2$, respectively. So, \eqref{2ransum} can be written as
 \begin{align*}
 d_W(T_n^{(1)},T_n^{(2)})& \leq c_1 \sum_{j=1}^{n}p_j \bigg|\bigg(\int_{0}^{\infty}\bigg(e^{-\alpha_j/\pi_{j}^{(1)} u}-e^{-\alpha_j/w_{j}^{(1)} u}\bigg)^2 du\bigg)^{\frac{1}{2}}\bigg|\\
 &\quad+c_2\sum_{j=1}^{n}q_j \bigg|\bigg(\int_{0}^{\infty}\bigg(e^{-\beta_j/\pi_{j}^{(2)} u}-e^{-\beta_j/w_{j}^{(2)} u}\bigg)^2 du\bigg)^{\frac{1}{2}}\bigg|\\
 &=c_1 \sum_{j=1}^{n}\frac{p_j}{\sqrt{\alpha_j}}\bigg|\frac{w_j^{(1)} + \pi_{j}^{(1)}}{2}-2\frac{w_j^{(1)}\pi_{j}^{(1)} }{w_j^{(1)}+\pi_{j}^{(1)}}  \bigg|^{1/2}\\
 &\quad\quad\quad + c_2\sum_{j=1}^{n}\frac{q_j}{\sqrt{\beta_j}}\bigg|\frac{w_j^{(2)}+\pi_{j}^{(2)}}{2}-2\frac{w_j^{(2)}\pi_{j}^{(2)} }{w_j^{(2)}+\pi_{j}^{(2)}}  \bigg|^{1/2}\\
 &=c_1 \sum_{j=1}^{n}\frac{p_j}{\sqrt{2\alpha_j}} \frac{ \big(w_j^{(1)} - \pi_{j}^{(1)}\big) }{\big(w_j^{(1)} + \pi_{j}^{(1)}\big)^{1/2}}+c_2\sum_{j=1}^{n}\frac{q_j}{\sqrt{2\beta_j}}\frac{\big(w_j^{(2)}-\pi_{j}^{(2)}\big)}{\big(w_j^{(2)}+\pi_{j}\big)^{1/2}}.
 \end{align*}
 This proves the result.	
 \end{proof}

\begin{cor}
	Under the assumption of Theorem \ref{twosumsthm}. Let $S_n^{(1)}= \sum_{j=1}^{n}w_j^{(1)}X_j$ and  $S_n^{(2)}= \sum_{j=1}^{n}\pi_j^{(1)}X_j$. If $\beta_j \to \infty$, as $n\to \infty$, then
	\begin{align*}
	d_W(S_n^{(1)},S_n^{(2)}) \leq c_1 \sum_{j=1}^{n}\frac{p_j}{\sqrt{2\alpha_j}} \frac{ \big(w_j^{(1)} - \pi_{j}^{(1)}\big) }{\big(w_j^{(1)} + \pi_{j}^{(1)}\big)^{1/2}},
	\end{align*}
	where $c_1$ is some positive constant.
\end{cor}
\begin{proof}
Note from \eqref{01cflcbg} that, if $\beta_j \to \infty$, as $n\to \infty$,
\begin{align*}
T_n^{(1)} \overset{\mathfrak{L}}{\to} S_n^{(1)} \text{ and }T_n^{(2)} \overset{\mathfrak{L}}{\to} S_n^{(2)}.
\end{align*}
Also, it follows \cite[Theorem 7.12]{Villani} that, if $T_n^{(1)} \overset{\mathfrak{L}}{\to} S_n^{(1)} \text{ and }T_n^{(2)} \overset{\mathfrak{L}}{\to} S_n^{(2)} \text{ as }\beta_j \to \infty,$
\begin{align*}
d_W(S_n^{(1)},S_n^{(2)})=\lim_{\beta_j \to \infty}d_W(T_n^{(1)},T_n^{(2)}).
\end{align*}
\noindent
Applying Theorem \ref{twosumsthm} and taking the limit as $\beta_j \to \infty$, we get from
\eqref{twosumbd}
\begin{align*}
d_W(S_n^{(1)},S_n^{(2)}) &\leq \lim_{\beta_j \to \infty} \bigg(c_1 \sum_{j=1}^{n}\frac{p_j}{\sqrt{2\alpha_j}} \frac{ \big(w_j^{(1)} - \pi_{j}^{(1)}\big) }{\big(w_j^{(1)} + \pi_{j}^{(1)}\big)^{1/2}}+c_2\sum_{j=1}^{n}\frac{q_j}{\sqrt{2\beta_j}}\frac{\big(w_j^{(2)}-\pi_{j}^{(2)}\big)}{\big(w_j^{(2)}+\pi_{j}\big)^{1/2}} \bigg) \\
 &=c_1 \sum_{j=1}^{n}\frac{p_j}{\sqrt{2\alpha_j}} \frac{ \big(w_j^{(1)} - \pi_{j}^{(1)}\big) }{\big(w_j^{(1)} + \pi_{j}^{(1)}\big)^{1/2}},
\end{align*}
\noindent
which proves the result.
\end{proof}
\begin{rem}
 Note that if $\alpha_j \to \infty,$ as $n\to \infty$, then $d_W(S_n^{(1)},S_n^{(2)}) \to 0.$
 Also, if $w_j^{(1)}=\pi_j^{(1)},$ as $n\to \infty$, then $d_W(S_n^{(1)},S_n^{(2)} ) \to 0$, as expected.
\end{rem}
\subsection{Limit of compound Poisson distributions}Next, we obtain, for the Kolmogorov distance $d_K$, the error bounds for a sequence of compound Poisson distributions that converges to the distribution of linear combination of BG r.v.s. Indeed, a more general result is shown to be true  in Theorem 1.2.18 of \cite{apple2009}, which shows an infinitely divisible distribution is a limit of compound Poisson distributions. Since $T_n$ is infinitely divisible, our result also follows from their result. However, Lemma \ref{approxthm1} provides the rate of convergence of this result. Before stating our result, let $Z_m$ be a compound Poisson r.v.s with cf (see \cite[Theorem 1.2.18]{apple2009})
	\begin{align}\label{cpdcf00}
	\phi_{m}(z):= \exp\bigg(m \left( \phi_{T_n}^{\frac{1}{m}}(z)-1  \right)   \bigg),~z\in\mathbb{{R}}, ~m\geq 1,
	\end{align}
	where $\phi_{T_n}(z)$ is given in \eqref{lcbg2}.  
\begin{lem}\label{approxthm1}
	Let $T_n$ be defined as in \eqref{defTn}. Then
	\begin{align}\label{Kbound1}
	d_{K}(Z_m,T_n) \leq c \bigg(\sum_{j=1}^{2}|C_{j}(T_n)|\bigg)^{\frac{2}{5}} \bigg(\frac{1}{m}\bigg)^{\frac{1}{5}},~m\geq 1,
	\end{align}
	where $c>0$ is some positive constant.
\end{lem} 

\vfill
\begin{proof}
	Let $b=\int_{-1}^{1}u\nu_{T_n}(du)$, where $\nu_{T_n}$ is defined in \eqref{lcbg1}. Then by Equation (2.6) of \cite{k0}, the distribution of $T_n$ is an infinitely divisible distribution with the triplet $b,0$ and $\nu_{T_n}$. Note that the distribution of $T_n$ is absolutely continuous with respect to the L\'evy measure with a bounded density and $\mathbb{E}|T_n|^2 < \infty$ (see \eqref{momentformula}). Recall from Proposition 4.11 of \cite{k0} that, if $T_n\sim ID(b,0,\nu_{T_n})$, that is, $T_n$ follows an infinitely divisible distribution with triplet $(b,0,\nu_{T_n})$ with cf $\phi_{T_n}$, and $Z_m$'s are compound Poisson r.v.s each of with cf as \eqref{cpdcf00}, then
	\begin{align}\label{kolbdd}
	d_{K}(Z_m,T_n)&\leq c\bigg( |\mathbb{E}[T_n]|+\int_{\mathbb{R}}u^2\nu_{T_n}(du)  \bigg)^{\frac{2}{p+4}}\bigg(\frac{1}{m}\bigg)^{\frac{1}{p+4}},
	\end{align}
	\noindent
	where $|\phi_{T_n}(z)| \displaystyle\int_{0}^{|z|}\frac{ds}{|\phi_{T_n}(s)|} \leq e^{c_0} |z|^p,$ $p\geq 1$.
	
	\noindent
	Now observe that
	\begin{align*}
		c_0:=\sup_{s\in\mathbb{R}}\bigg|\displaystyle\int_{\mathbb{{R}}}(e^{isu} -1)\nu_{T_n}(du)  \bigg|< \infty.
	\end{align*}
	\noindent
	So,
	\begin{align*}
	|\phi_{T_n}(z)| \displaystyle\int_{0}^{|z|}\frac{ds}{|\phi_{T_n}(s)|}
	&\leq e^{c_0} \displaystyle\int_{0}^{|z|}ds\\
	&=e^{c_0}|z|.
	\end{align*}
	\noindent
	Hence by \eqref{kolbdd}, for $p=1$, we get
	\begin{align*}
	d_{K}(Z_m,T_n)&\leq c\bigg( |\mathbb{E}[T_n]|+\int_{\mathbb{R}}u^2\nu_{T_n}(du)  \bigg)^{\frac{2}{5}}\bigg(\frac{1}{m}\bigg)^{\frac{1}{5}}\\
	&= c \bigg(\sum_{j=1}^{2}|C_{j}(T_n)|\bigg)^{\frac{2}{5}} \bigg(\frac{1}{m}\bigg)^{\frac{1}{5}}, 
	\end{align*}
	\noindent
	$\text{ since }\int_{\mathbb{R}}u^2\nu_{T_n}(du)=C_2(T_n).$ This proves the result.
\end{proof}
\begin{rem}
Note that if $m \to \infty$, $d_{K}(Z_m,T_n)\to 0$, and so the distribution of $T_n$ is the limit of a compound Poisson distributions. 
\end{rem}
\vspace*{-0.7cm}
\subsection{Approximation for bilateral gamma distributions}Our next result presents an error bound for the distributional approximation of linear combination of independent BG r.v.s to a BG distribution. We refer the reader to \cite{dalygaunt2016} and \cite{leyreisw2017} for a number of similar bounds for comparisons of univariate distributions. 
\noindent
Let $f$ denote the solution of the Stein equation  \eqref{PP2:e15}. The proof of Theorem \ref{ApproxTHM} relies on the bounds for $|xf^{(k)}(x)|$, $k = 1, 2$, established in Theorem \ref{th3}. Define $g_n:=\prod_{j=1}^{n}\alpha_j \beta_j$,  $h_n:=g_n\sum_{j=1}^{n} (w_j^{(1)}w_j^{(2)}+|w_j^{(1)}\beta_j-w_j^{(2)}\alpha_j |  )/\alpha_j\beta_j$ and
$\kappa_n:=\frac{g_n}{g_n-h_n}.$
\vspace*{-0.4cm}
\begin{thm}\label{ApproxTHM}
Let $T_n$ be defined as in \eqref{defTn} such that $g_n> h_n$.
Also, let $Z_{bg} \sim BG(\alpha,p,\beta,q) $. Then,
\begin{align}\label{stsdbd}
\nonumber	d_{3}(T_n,Z_{bg}) & \leq \left(2+\frac{1}{3}|\mathbb{E}(T_n)|\right)\kappa_n \bigg|\sum_{j=1}^{n}\frac{w_{j}^{(1)}w_{j}^{(2)} }{\alpha_j \beta_j}- \frac{1}{\alpha \beta}\bigg|\\
\nonumber&\quad\quad+\left(2+\frac{1}{2}|\mathbb{E}(T_n)|\right)\kappa_n\bigg| \sum_{j=1}^{n}\left(\frac{w_{j}^{(1)} }{\alpha_j}-\frac{w_{j}^{(2)}}{\beta_j} \right)  -\bigg(\frac{1}{\alpha}-\frac{1}{\beta}\bigg) \bigg|\\
&\quad\quad\quad+\frac{1}{2}\kappa_n\left| \sum_{j=1}^{n}\frac{w_{j}^{(1)}w_{j}^{(2)}(p_j+q_j)}{\alpha_j \beta_j} -\frac{p+ q}{\alpha \beta} \right|+\kappa_n\left|\mathbb{E}(T_n) -\mathbb{E}(Z_{bg}) \right|. 
	\end{align}
\end{thm}
\vspace{-1.12cm}
\begin{proof}  
Let $(f,h)$ be the solution pair of \eqref{PP2:e15}. Then replacing $x$ by $Z_{bg} \sim BG(\alpha,p,\beta,q)$ in the Stein equation and taking expectation, we get
	\begin{align}
	\nonumber	\mathbb{E}(h(Z_{bg}))-\mathbb{E}(h(T_n))=&\mathbb{E}\left(-Z_{bg}f(Z_{bg})+\displaystyle\int_{\mathbb{R}}f(Z_{bg}+u)u\nu_{T_n}(du)  \right)\\
	\nonumber=&\mathbb{E} \bigg(-Z_{bg}f(Z_{bg})+\sum_{j=1}^{n}\bigg(p_j \displaystyle\int_{0}^{\infty}f(Z_{bg}+u)e^{-\alpha_j/w_{j}^{(1)} u}du\\
	\nonumber	
	&\quad -  q_j \displaystyle\int_{0}^{\infty}f(Z_{bg}-u)e^{-\beta_j/w_{j}^{(2)} u}du \bigg)\bigg)~~\text{(split the  measure $\nu_{T_n}$)}\\
	\nonumber=& \mathbb{E}\bigg( \left(  \sum_{j=1}^{n}\left(\frac{w_{j}^{(1)}p_j}{\alpha_j}-\frac{w_{j}^{(2)}q_j}{\beta_j}   \right) -Z_{bg} \right)f(Z_{bg})\bigg)\\
\nonumber	&+\sum_{j=1}^{n}\frac{w_{j}^{(1)}p_j}{\alpha_j}\mathbb{E} \left(\displaystyle\int_{0}^{\infty}f^{\prime}(Z_{bg}+u)e^{-\alpha_j/w_{j}^{(1)} u} du \right)\\
	\label{ibpf1}&\quad\quad+\sum_{j=1}^{n}\frac{w_{j}^{(2)}q_j}{\beta_j}\mathbb{E} \left(\displaystyle\int_{0}^{\infty}f^{\prime}(Z_{bg}-u)e^{-\beta_j/w_{j}^{(2)} u} du \right), 
	\end{align}
	\noindent
	where the last equality follows by the integration by parts formula. Note that
	\begin{align}\label{ibpf2}
	\nonumber&\sum_{j=1}^{n}\mathbb{E}\bigg(p_j \int_{0}^{\infty}f^{\prime}(Z_{bg}+u)e^{-\alpha_j/w_{j}^{(1)} u}du-	q_j\int_{0}^{\infty}f^{\prime}(Z_{bg}-u)e^{-\beta_j/w_{j}^{(2)} u}du\bigg)\\
	&\quad\quad\quad=\mathbb{E}\bigg(\displaystyle\int_{\mathbb R}uf^{\prime}(Z_{bg}+u)\nu_{T_n}(du)\bigg).
	\end{align}
	Using \eqref{ibpf2} in \eqref{ibpf1}, we get
	\begin{align}
	\nonumber\mathbb{E}(h(Z_{bg}))-\mathbb{E}(h(T_n))=&\mathbb{E} \bigg(\left(  \mathbb{E}(T_n) -Z_{bg} \right)f(Z_{bg})\bigg)\\
\nonumber	&+\sum_{j=1}^{n}\bigg(\left(\frac{w_{j}^{(1)} }{\alpha_j}-\frac{w_{j}^{(2)}}{\beta_j} \right)\mathbb{E} \bigg( p_j \int_{0}^{\infty}f^{\prime}(Z_{bg}+u)e^{-\alpha_j/w_{j}^{(1)} u}du\\
	\nonumber&\quad-	q_j\int_{0}^{\infty}f^{\prime}(Z_{bg}-u)e^{-\beta_j/w_{j}^{(2)} u}du \bigg)\\
	\nonumber&\quad\quad\quad+\frac{w_{j}^{(2)}p_j}{\beta_j}\mathbb{E} \left(\displaystyle\int_{0}^{\infty}f^{\prime}(Z_{bg}+u)e^{-\alpha_j/w_{j}^{(1)} u} du \right)\\
	\nonumber&\quad\quad\quad\quad\quad\quad+\frac{w_{j}^{(1)}q_j}{\alpha_j}\mathbb{E} \left(\displaystyle\int_{0}^{\infty}f^{\prime}(Z_{bg}-u)e^{-\beta_j/w_{j}^{(2)} u} du \right)\bigg)\\
	\nonumber=&\mathbb{E} \bigg(\left(  \mathbb{E}(T_n) -Z_{bg} \right)f(Z_{bg})\bigg)\\
\nonumber	&+\sum_{j=1}^{n}\left(\frac{w_{j}^{(1)} }{\alpha_j}-\frac{w_{j}^{(2)}}{\beta_j} \right)\mathbb{E} \left( \int_{\mathbb{R}} f^{\prime}(Z_{bg}+u) u\nu_{T_n}(du) \right)\\
	\nonumber&\quad\quad\quad+\sum_{j=1}^{n}\frac{w_{j}^{(2)}p_j}{\beta_j}\mathbb{E} \left(\displaystyle\int_{0}^{\infty}f^{\prime}(Z_{bg}+u)e^{-\alpha_j/w_{j}^{(1)} u} du \right)\\
	\label{ibpf3}&\quad\quad\quad\quad\quad\quad+\sum_{j=1}^{n}\frac{w_{j}^{(1)}q_j}{\alpha_j}\mathbb{E} \left(\displaystyle\int_{0}^{\infty}f^{\prime}(Z_{bg}-u)e^{-\beta_j/w_{j}^{(2)} u} du \right).
	\end{align}
	\noindent
	Applying integration by parts formula to the last two terms of \eqref{ibpf3}, we get	
	\begin{align}
	\nonumber\mathbb{E}(h(Z_{bg}))-\mathbb{E}(h(T_n))=&\mathbb{E} \bigg(\left(  \mathbb{E}(T_n) -Z_{bg} \right)f(Z_{bg})\bigg)\\
\nonumber	&\quad+\sum_{j=1}^{n}\bigg(\left(\frac{w_{j}^{(1)} }{\alpha_j}-\frac{w_{j}^{(2)}}{\beta_j} \right)\mathbb{E} \left( \int_{\mathbb{R}} f^{\prime}(Z_{bg}+u) u\nu_{T_n}(du) \right)\\
	\nonumber&\quad\quad+\frac{w_{j}^{(1)}w_{j}^{(2)}(p_j+q_j)}{\alpha_j \beta_j}\mathbb{E}\left(f^{\prime}(Z_{bg}) \right)\\
	\nonumber&\quad\quad+	\frac{w_{j}^{(1)}w_{j}^{(2)}}{\alpha_j \beta_j}\mathbb{E}\bigg(p_j \int_{0}^{\infty}f^{\prime\prime}(Z_{bg}+u)e^{-\alpha_j/w_{j}^{(1)} u}du\\
	\nonumber&\quad\quad\quad\quad-q_j\int_{0}^{\infty}f^{\prime\prime}(Z_{bg}-u)e^{-\beta_j/w_{j}^{(2)} u}du\bigg)\bigg)\\
	\nonumber =&\mathbb{E} \bigg(\left(  \mathbb{E}(T_n) -Z_{bg} \right)f(Z_{bg})\bigg)\\
	\nonumber	&\quad+\sum_{j=1}^{n}\bigg(\left(\frac{w_{j}^{(1)} }{\alpha_j}-\frac{w_{j}^{(2)}}{\beta_j} \right)\mathbb{E} \left( \int_{\mathbb{R}} f^{\prime}(Z_{bg}+u) u\nu_{T_n}(du) \right)\\
	\nonumber&\quad\quad+\frac{w_{j}^{(1)}w_{j}^{(2)}(p_j+q_j)}{\alpha_j \beta_j}\mathbb{E}\left(f^{\prime}(Z_{bg}) \right)\\
&\quad\quad+	\frac{w_{j}^{(1)}w_{j}^{(2)}}{\alpha_j \beta_j}\mathbb{E}\bigg(\displaystyle\int_{\mathbb{{R}}}f^{\prime\prime}(Z_{bg}+u)u\nu_{T_n}(du) \bigg)\bigg).\label{tt0}
	\end{align}
	\noindent
	Let $h_1,h_2 \in \mathcal{W}_3$ be two associated test functions with $f^\prime$ and $f^{\prime\prime}$ satisfy the Stein equation \eqref{PP2:e15}. Then from \eqref{tt0}, we get
	\begin{align}
	\nonumber\mathbb{E}(h(Z_{bg}))-\mathbb{E}(h(T_n)) =&\mathbb{E}\bigg(	\sum_{j=1}^{n}\frac{w_{j}^{(1)}w_{j}^{(2)}}{\alpha_j \beta_j}Z_{bg}f^{\prime\prime}(Z_{bg})+ \sum_{j=1}^{n}\bigg(\frac{w_{j}^{(1)}w_{j}^{(2)}(p_j+q_j)}{\alpha_j \beta_j}Z_{bg}\\
\nonumber	&+ \left(\frac{w_{j}^{(1)} }{\alpha_j}-\frac{w_{j}^{(2)}}{\beta_j} \right)\bigg)f^{\prime}(Z_{bg})+\left(  \mathbb{E}(T_n) -Z_{bg} \right)f(Z_{bg})  \bigg)\\
	\nonumber&\quad +\sum_{j=1}^{n}\left(\frac{w_{j}^{(1)} }{\alpha_j}-\frac{w_{j}^{(2)}}{\beta_j} \right)\bigg(\mathbb{E}(h_1(Z_{bg}))-\mathbb{E}(h_1(T_n))\bigg)\\
	&\quad\quad\quad+ \sum_{j=1}^{n}\frac{w_{j}^{(1)}w_{j}^{(2)}}{\alpha_j \beta_j} \bigg(\mathbb{E}(h_2(Z_{bg}))-\mathbb{E}(h_2(T_n))\bigg).\label{ibpf5}
	\end{align}
	\noindent
	Taking supremum over the functions $h \in \mathcal{W}_3,$ we set
\begin{align}
	\nonumber d_3(T_n,Z_{bg}) \leq ~&\mathbb{E}\bigg(	\sum_{j=1}^{n}\frac{w_{j}^{(1)}w_{j}^{(2)}}{\alpha_j \beta_j}Z_{bg}f^{\prime\prime}(Z_{bg})+ \sum_{j=1}^{n}\bigg(\frac{w_{j}^{(1)}w_{j}^{(2)}(p_j+q_j)}{\alpha_j \beta_j}Z_{bg}\\
	\nonumber	&+ \left(\frac{w_{j}^{(1)} }{\alpha_j}-\frac{w_{j}^{(2)}}{\beta_j} \right)\bigg)f^{\prime}(Z_{bg})+\left(  \mathbb{E}(T_n) -Z_{bg} \right)f(Z_{bg})  \bigg)\\
	\nonumber&\quad +\sum_{j=1}^{n}\left|\frac{w_{j}^{(1)} }{\alpha_j}-\frac{w_{j}^{(2)}}{\beta_j} \right| d_3(T_n,Z_{bg})\\
	&\quad\quad\quad+ \sum_{j=1}^{n}\frac{w_{j}^{(1)}w_{j}^{(2)}}{\alpha_j \beta_j}  d_3(T_n,Z_{bg}).\label{ibpf7}
	\end{align}
	\noindent
	That is,
	\begin{align}
	\nonumber&\left(1 - \sum_{j=1}^{n}\frac{|w_{j}^{(1)}\beta_j-w_{j}^{(2)}\alpha_j|+w_j^{(1)} w_j^{(2)}}{\alpha_j \beta_j} \right)d_{3}(T_n,Z_{bg})\\	
\nonumber	&\quad \leq\bigg| \mathbb{E}\bigg(	\sum_{j=1}^{n}\frac{w_{j}^{(1)}w_{j}^{(2)}}{\alpha_j \beta_j}Z_{bg}f^{\prime\prime}(Z_{bg})+ \sum_{j=1}^{n}\bigg(\frac{w_{j}^{(1)}w_{j}^{(2)}(p_j+q_j)}{\alpha_j \beta_j}Z_{bg}\\
	&\quad\quad+ \left(\frac{w_{j}^{(1)} }{\alpha_j}-\frac{w_{j}^{(2)}}{\beta_j} \right)\bigg)f^{\prime}(Z_{bg})+\left(  \mathbb{E}(T_n) -Z_{bg} \right)f(Z_{bg})  \bigg)\bigg|.\label{tt1}
	\end{align}
\noindent
Let $\mathcal{A}_{bg}$ denote a Stein operator of $Z_{bg}\sim BG(\alpha,p,\beta,q)$. Then
\begin{align}
	\mathbb{E}\left(\mathcal{A}_{bg}f(Z_{bg}) \right)=0,~f\in \mathcal{S}(\mathbb{{R}}),
\end{align}	
where $\mathcal{A}_{bg}f(x)=\frac{1}{\alpha \beta}xf^{\prime\prime}(x)+ \left(\frac{p+q}{\alpha\beta}+ \left(\frac{1}{\alpha}-\frac{1}{\beta} \right)x   \right)f^{\prime}(x)+\left(  \mathbb{E}(Z_{bg}) -x \right)f(x)$, see \cite[Example 5.3]{BUV1}. Then, from \eqref{tt1},
\begin{align}
\nonumber&\left(1 - \sum_{j=1}^{n}\frac{|w_{j}^{(1)}\beta_j-w_{j}^{(2)}\alpha_j|+w_j^{(1)} w_j^{(2)}}{\alpha_j \beta_j} \right)d_{3}(T_n,Z_{bg})\\	
\nonumber	&\quad \leq\bigg| \mathbb{E}\bigg(	\sum_{j=1}^{n}\frac{w_{j}^{(1)}w_{j}^{(2)}}{\alpha_j \beta_j}Z_{bg}f^{\prime\prime}(Z_{bg})+ \sum_{j=1}^{n}\bigg(\frac{w_{j}^{(1)}w_{j}^{(2)}(p_j+q_j)}{\alpha_j \beta_j}Z_{bg}\\
\nonumber	&\quad\quad+ \left(\frac{w_{j}^{(1)} }{\alpha_j}-\frac{w_{j}^{(2)}}{\beta_j} \right)\bigg)f^{\prime}(Z_{bg})+\left(  \mathbb{E}(T_n) -Z_{bg} \right)f(Z_{bg})  \bigg)- \mathbb{E}\bigg(\frac{1}{\alpha \beta}Z_{bg}f^{\prime\prime}(Z_{bg})\\
\nonumber&\quad\quad\quad\quad+ \left(\frac{p+q}{\alpha\beta}+ \left(\frac{1}{\alpha}-\frac{1}{\beta} \right)Z_{bg}   \right)f^{\prime}(Z_{bg})+\left(  \mathbb{E}(Z_{bg}) -Z_{bg} \right)f(Z_{bg})  \bigg) \bigg|\\
\nonumber&\quad=\bigg|\mathbb{E}\bigg( \left( \sum_{j=1}^{n}\frac{w_{j}^{(1)}w_{j}^{(2)} }{\alpha_j \beta_j}- \frac{1}{\alpha \beta} \right)Z_{bg}f^{\prime\prime}(Z_{bg}) \\
\nonumber &\quad+ \left(\sum_{j=1}^{n}\left(\frac{w_{j}^{(1)} }{\alpha_j}-\frac{w_{j}^{(2)}}{\beta_j} \right)  -\bigg(\frac{1}{\alpha}-\frac{1}{\beta}\bigg)\right)Z_{bg}f^{\prime}(Z_{bg}) \bigg) \\
\nonumber&\quad\quad+\left( \sum_{j=1}^{n}\frac{w_{j}^{(1)}w_{j}^{(2)}(p_j+q_j)}{\alpha_j \beta_j} -\frac{p+ q}{\alpha \beta} \right)f^{\prime}(Z_{bg})+ \left(\mathbb{E}(T_n) -\mathbb{E}(Z_{bg}) \right)f(Z_{bg}) \bigg|\\
\nonumber& \leq \bigg|\sum_{j=1}^{n}\frac{w_{j}^{(1)}w_{j}^{(2)} }{\alpha_j \beta_j}- \frac{1}{\alpha \beta}\bigg|\|xf^{\prime\prime}(x)\|+\bigg| \sum_{j=1}^{n}\left(\frac{w_{j}^{(1)} }{\alpha_j}-\frac{w_{j}^{(2)}}{\beta_j} \right)  -\bigg(\frac{1}{\alpha}-\frac{1}{\beta}\bigg) \bigg|\|xf^{\prime}(x)\|\\
\nonumber &\quad\quad\quad+\left| \sum_{j=1}^{n}\frac{w_{j}^{(1)}w_{j}^{(2)}(p_j+q_j)}{\alpha_j \beta_j} -\frac{p+ q}{\alpha \beta} \right|\|f^{\prime}\|+\left|\mathbb{E}(T_n) -\mathbb{E}(Z_{bg}) \right|\|f\|,
\end{align}
\noindent
hence,
\begin{align}\label{lastbd}
\nonumber d_3(T_n,Z_{bg})&\leq \kappa_n \bigg(\bigg|\sum_{j=1}^{n}\frac{w_{j}^{(1)}w_{j}^{(2)} }{\alpha_j \beta_j}- \frac{1}{\alpha \beta}\bigg|\|xf^{\prime\prime}(x)\|+\bigg| \sum_{j=1}^{n}\left(\frac{w_{j}^{(1)} }{\alpha_j}-\frac{w_{j}^{(2)}}{\beta_j} \right) \\
\nonumber&\quad\quad
 -\bigg(\frac{1}{\alpha}-\frac{1}{\beta}\bigg) \bigg|\|xf^{\prime}(x)\|+\left| \sum_{j=1}^{n}\frac{w_{j}^{(1)}w_{j}^{(2)}(p_j+q_j)}{\alpha_j \beta_j} -\frac{p+ q}{\alpha \beta} \right|\|f^{\prime}\|\\
 &\quad\quad\quad\quad+\left|\mathbb{E}(T_n) -\mathbb{E}(Z_{bg}) \right|\|f\|  \bigg).
\end{align}
\noindent
Using the estimates of Theorem \ref{th3} to bound \eqref{lastbd} yields \eqref{stsdbd}.
\end{proof}
\begin{rem}
 Note that if $w_j^{(1)}=w_j^{(2)}=1$, $n=1$, $\alpha_j \to \alpha,$ $\beta_j \to \beta$, $p_j \to p$ and $q_j \to q$, then by \eqref{stsdbd}, $d_3(T_1,Z_{bg}) \to 0$, as expected.
\end{rem}
\vspace{-0.5cm}
\noindent
We now have the following corollary for VG approximation to the linear combinations of independent BG r.v.s. The VG approximation for sums of independent r.v.s has been widely studied in the literature; see, for instance, \cite{Gaunt2014, Gaunt2020, Gaunt2022}.
\vspace{-0.5cm}
\begin{cor}\label{ApproxTHM1}
Let $T_n$ be defined as in \eqref{defTn} such that $g_n> h_n$.
Also let $Z_{vg} \sim VG (\alpha,\beta,p)$. Then,
\begin{align}\label{stsdbdvg}
\nonumber d_{3}(T_n,Z_{vg}) & \leq \left(2+\frac{1}{3}|\mathbb{E}(T_n)|\right)\kappa_n\bigg|\sum_{j=1}^{n}\frac{w_{j}^{(1)}w_{j}^{(2)} }{\alpha_j \beta_j}- \frac{1}{\alpha \beta}\bigg|\\
\nonumber &\quad\quad+\left(2+\frac{1}{2}|\mathbb{E}(T_n)|\right)\kappa_n\bigg| \sum_{j=1}^{n}\left(\frac{w_{j}^{(1)} }{\alpha_j}-\frac{w_{j}^{(2)}}{\beta_j} \right)  -\bigg(\frac{1}{\alpha}-\frac{1}{\beta}\bigg) \bigg|\\
&\quad\quad\quad+\frac{1}{2}\kappa_n\left| \sum_{j=1}^{n}\frac{w_{j}^{(1)}w_{j}^{(2)}(p_j+q_j)}{\alpha_j \beta_j} -\frac{2p}{\alpha \beta} \right|+\kappa_n\left|\mathbb{E}(T_n) -\mathbb{E}(Z_{vg}) \right|.
\end{align}
\end{cor}
\begin{rem}
Note that if $w_j^{(1)}=w_j^{(2)}=1$, $p_j=q_j=p_n$, $\alpha_j=\beta_j=\sqrt{n} \alpha$, $\alpha=\beta$, and $Var(T_n) \to Var (Z_{vg})$, then by \eqref{stsdbdvg}, $d_3(T_n,Z_{vg}) \to 0$ as $n \to \infty$. 
\end{rem}
\vspace{-0.5cm}
\noindent
The following example establishes an explicit error bound for the normal approximation of linear combinations of BG r.v.s.
\vspace{-0.6cm}
\begin{exmp}
	Let $Z_p\sim SVG(\sqrt{2p}/\sigma, p)$ and $Z_\sigma \sim \mathcal{N}(0,\sigma^2) $. Recall from Section \ref{prebg} that,
$SVG(\sqrt{2p}/\sigma, p) \overset{d}{=}BG(\sqrt{2p}/\sigma, p,\sqrt{2p}/\sigma, p)$. Then, the cf of $SVG(\sqrt{2p}/\sigma, p)$ is
\begin{align}
\label{svgcf000}	\phi_{sv}(z)&= \bigg(1+ \frac{z^2 \sigma^2}{2p} \bigg)^{-p},~z\in\mathbb{R}\\
&=\exp\bigg(\displaystyle\int_{\mathbb{R}} (e^{izu}-1)\nu_{sv}(du)  \bigg),~z\in\mathbb{{R}},
\end{align}  
\noindent
where the L\'evy measure $\nu_{sv}$ is 
\begin{align*}
\nu_{sv}(du)=\bigg(\frac{p}{u}e^{- \frac{\sqrt{2p}}{\sigma} u}\textbf{1}_{(0,\infty)}(u)+ \frac{p}{|u|}e^{- \frac{\sqrt{2p}}{\sigma} |u|}\textbf{1}_{(-\infty,0)}(u) \bigg).
\end{align*}
\noindent
Note from \eqref{svgcf000} that, 
\begin{align*}
\lim_{p\to\infty}\phi_{sv}(z)=e^{-\frac{\sigma^2 z^2}{2}}.
\end{align*}
\noindent
That is, $Z_p \overset{L}{\to} Z_\sigma \sim \mathcal{N}(0,\sigma^2),$ as $p\to\infty.$ Also, it follows \cite[Theorem 7.12]{Villani} that, if $Z_p \overset{L}{\to} Z_\sigma,$ as $p \to \infty$,

\begin{align}\label{svgcf1100}
d_{3}(T_n,Z_\sigma)= \lim_{p \to \infty}d_{3}(T_n,Z_p).
\end{align}
\noindent
Applying Theorem \ref{ApproxTHM} to $Z_{bg} = Z_p$, and taking the limit as $p\to \infty$, we get from \eqref{stsdbd}
	\begin{align}\label{bgclt}
\nonumber	d_{3}(T_n,Z_\sigma) & \leq \lim_{p\to\infty}\bigg(  \left(2+\frac{1}{3}|\mathbb{E}(T_n)|\right)\kappa_n \bigg|\sum_{j=1}^{n}\frac{w_{j}^{(1)}w_{j}^{(2)} }{\alpha_j \beta_j}- \frac{\sigma^2}{2p}\bigg|\\
\nonumber &\quad\quad+\left(2+\frac{1}{2}|\mathbb{E}(T_n)|\right)\kappa_n\bigg| \sum_{j=1}^{n}\left(\frac{w_{j}^{(1)} }{\alpha_j}-\frac{w_{j}^{(2)}}{\beta_j} \right)\bigg|\\
\nonumber &\quad\quad\quad+\frac{1}{2}\kappa_n\left| \sum_{j=1}^{n}\frac{w_{j}^{(1)}w_{j}^{(2)}(p_j+q_j)}{\alpha_j \beta_j} -\sigma^2 \right|+\kappa_n\left|\mathbb{E}T_n\right|\bigg)\\
\nonumber &= \left(2+\frac{1}{3}|\mathbb{E}(T_n)|\right)\kappa_n \sum_{j=1}^{n}\frac{w_{j}^{(1)}w_{j}^{(2)} }{\alpha_j \beta_j}\\
\nonumber &\quad\quad+\left(2+\frac{1}{2}|\mathbb{E}(T_n)|\right)\kappa_n\bigg| \sum_{j=1}^{n}\left(\frac{w_{j}^{(1)} }{\alpha_j}-\frac{w_{j}^{(2)}}{\beta_j} \right)\bigg|\\
&\quad\quad\quad+\frac{1}{2}\kappa_n\left| \sum_{j=1}^{n}\frac{w_{j}^{(1)}w_{j}^{(2)}(p_j+q_j)}{\alpha_j \beta_j} -\sigma^2 \right|+\kappa_n\left|\mathbb{E}T_n\right|,
\end{align}
which gives the error in the closeness between the distributions of $T_n$ and $Z_\sigma$. 
\end{exmp}
\begin{rem}
(i) Note that if we set $w_j^{(1)}=w_j^{(2)}=w_j$, $p_j=q_j$, $\alpha_j=\beta_j$ such that $w_j/\alpha_j,w_j/\beta_j \to 0$ and $Var(T_n) \to \sigma
^2$, then by \eqref{bgclt}, $d_3(T_n, Z_\sigma) \to 0,$ as $n \to \infty$.

\item [(ii)]Note also that, if $\sigma=1, w_j^{(1)}=w_j^{(2)}=1/\sqrt{n}, \alpha_j=\beta_j=n\alpha , \text{ and } p_j=q_j=\frac{1}{2}n^3\alpha^2$ in \eqref{bgclt}, we get $d_3(T_n,Z_1) \to 0$ as $n\to \infty$. This result is, in a sense a generalization of the Proposition 3.1 of \cite{bilateral0}, since $T_1=w_1^{(1)}X_1-w_1^{(2)}Y_1\sim $ BG distribution.
\end{rem}

\section{Mixed Bilateral gamma process}\label{BGpro}
\noindent
In this section, we first introduce the mixed bilateral gamma process, which is a BG process with random parameters. We then discuss some of its important properties.

\noindent
Let $\{T_n(t)\}_{t\geq 0}$ be a L\'evy process with cf 
\begin{align}\label{cfBGRP}
	\mathbb{E}(e^{izT_n(t)})=e^{t \tau (z)},~z\in\mathbb{{R}},
\end{align}
\noindent
where $\tau(z)$ is the characteristic exponent, given by $\tau(z)=\int_{\mathbb{{R}}}(e^{izu}-1)\nu_{T_n}(du)$ and $\nu_{T_n}$ is the L\'evy measure given in \eqref{lcbg1}. Note the case $n=1$ was first considered by K$\ddot{\text{u}}$chler and Tappe \cite{bilateral0}. Recently, by incorporating a diffusion component into the BG process, Kirkby {\it et al.} \cite{bgm2024} introduced BG motion and applied it in stock models. From \eqref{lcbg1}, we see that $\nu_{T_n}(\mathbb{R})=\infty$ and $\int_{\mathbb{R}}|u| \nu_{T_n}(du)<\infty$. Since the Gaussian component is zero, $T_n(t)$ is classified as type B according to Definition 11.9 of Sato \cite{sato}. Moreover, $T_n(t)$ is a finite-variation process (see \cite[Theorem 21.9]{sato}) with countable jumping times, almost surely (see \cite[Theorem 21.3]{sato}). We note that all the increments of $T_n(t)$ have the BG distribution, but with random parameters. More precisely, 
\begin{align}\label{bgprorandom}
T_n(t) -T_n(s)\sim BG\left(\eta, (L_n+p)(t-s),\xi,(M_n+q)(t-s)\right),~0\leq s \leq t,	
\end{align}
\noindent
 where $\eta=\alpha/1-\alpha$, $\xi=\beta/1-\beta$, and $\alpha,\beta,L_n,M_n,p$, and $q$ are as defined in \eqref{ref1} and \eqref{ref2}, respectively. Also,
 \begin{align}\label{lcbgpro1}
T_n(t) \sim BG\left(\eta, (L_n+p)t,\xi,(M_n+q)t\right),~ t\geq 0.	
\end{align}
\noindent
Since $L_n$ and $M_n$ are random,  $T_n(t)$ is indeed a mixed BG process.

 \noindent
 Next, we present an application of the above introduced process to stock models.
 \vspace{-0.5cm}
 \subsection{Arbitrage-free Stock models}
 In the continuous time finance, one of the crucial task is to find a realistic and analytically tractable models for price evolutions of risky financial assets (see, \cite{bgm2024} and \cite{bilateral0}, \cite{bilateral2}). We consider the exponential L\'evy models (see, K$\ddot{\text{u}}$chler and Tappe \cite{bilateral2,tsd2013})
  \begin{align}\label{LCBGstock}
 	\begin{cases}
 	S_t&=S_0e^{T_n(t)}\\
 	B_t&=e^{rt},
 	\end{cases}
 \end{align}
\noindent
which consists two financial assets $(S,B)$. Here $S$ is the dividend paying stock with deterministic initial value $S_0$ with dividend rate $v\geq 0$. Also, $B$ is the bank account with fixed interest rate $r\geq 0$.  Note that the model \eqref{LCBGstock} is a generalization of BG stock models considered in \cite{bilateral2}. Recently, by incorporating a diffusion component into the BG process, Kirkby {\it et al.} \cite{bgm2024} developed a generalized exponential L\'evy model and derived an arbitrage-free option pricing formula. In practice, it is often necessary to address the appropriate pricing of European options $\Phi(S_T)$, where $T > 0$ represents the maturity time and $\Phi : \mathbb{R} \to \mathbb{R}$ denotes the payoff function. Also, the option price is given by $\pi:=e^{-rT}\mathbb{E}_{\mathbb{Q}}[\Phi(S_T)],$ where $\mathbb{Q} \sim \mathbb{P}$ is a local martingale measure. Here, $\mathbb{Q}$ is also a probability measure which is equivalent to the objective probability measure $\mathbb{P}$, such that the discounted stock price process
\begin{align}\label{dspp}
	\tilde{S}_t := e^{-(r-v)t}S_t = S_0 e^{T_n(t) - (r-v)t}, \quad t \ge 0
\end{align}
is a local $\mathbb{Q}$-martingale, see \cite{bilateral2}. 

\noindent
The following lemma follows from Lemma 3.1 of \cite{bilateral2}, which establishes necessary and sufficient conditions for the existence of a probability measure $\mathbb{Q} \sim \mathbb{P}$.
\vspace{-.5cm}
\begin{lem}
Let $r\geq v \geq 0$. Let $\{T_n(t)\}_{t\geq 0}$ be a L\'evy process with cf \eqref{cfBGRP}. Then, $\mathbb{P}$ is a martingale measure for $\tilde{S}_t$ if and only if 
\begin{align}\label{dspp1}
\mathbb{E}_{L_n} \left(\left(\frac{\eta}{\eta-1}\right)^{L_n}  \right) \mathbb{E}_{M_n} \left(\left(\frac{\xi}{\xi-1}\right)^{M_n}  \right)=(1-1/\eta)^p (1-1/\xi)^q e^{r-v},
\end{align}
where $\eta=\alpha/1-\alpha$, $\xi=\beta/1-\beta$, and $\alpha,\beta,L_n,M_n,p$, and $q$ are as defined in \eqref{ref1} and \eqref{ref2}, respectively.
\end{lem}
\vspace{-.8cm}
\begin{proof}
	From \eqref{mggd1}, we have $\mathbb{E}(e^{T_n(1)})=\mathbb{E}_{L_n} \left(\left(\frac{\eta}{\eta-1}\right)^{p+L_n}  \right) \mathbb{E}_{M_n} \left(\left(\frac{\xi}{\xi-1}\right)^{q+M_n}  \right)<\infty$. By Lemma 2.6 of \cite{bilateral2}, the discounted process $\tilde{S_t}$ in \eqref{dspp} is a local martingale if and only if $\mathbb{E}[e^{T_n(1)-(r-v)}]=1,$ which is only the case if and only if \eqref{dspp1} holds.
\end{proof}
\vspace{-.46cm}
\subsubsection{Pricing formula}
The existence of a martingale measure $\mathbb{Q} \sim \mathbb{P}$ ensures that the stock market is free of arbitrage, and the price of an European option $\Phi(S_T )$, where $T( > 0)$ is the time of maturity and $\Phi : \mathbb{R}\to \mathbb{R}$ is the pay-off function, is given by
\begin{align}
\pi=e^{-rT}\mathbb{E}_{\mathbb{Q}}[\Phi(S_T)].
\end{align}
\noindent
Also, an arbitrage-free pricing formula for a European Call Option at time $t (\ge 0)$ is, provided that $S_t = s$, given by (see \cite{bilateral0})
\begin{equation}\label{oppf1}
\pi = e^{-rT}\mathbb{E}_{\mathbb{Q}}\left[(S_T - K)^+ \mid S_t = s\right], ~T>t,
\end{equation}
where $K$ denotes the strike price. Setting $t^\prime=(T-t)$, we can express the expectation in \eqref{oppf1} as
\begin{align}\label{pf1}
\nonumber\pi&=e^{-rT}\mathbb{E}_{\mathbb{Q}}\left[(S_t e^{T_n(T)-T_n(t)} - K)^+ \mid S_t = s\right]\\
&= e^{-rT} \int_{\ln\left(\frac{K}{s}\right)}^{\infty} 
(se^x - K) h_{T_n}(x,t^\prime) dx,
\end{align}  
\noindent
where $h_{T_n}(x,t^\prime)$ denotes the pdf of $T_n(t^\prime) \sim BG \left(\eta, (L_n+p)t^\prime,\xi,(M_n+q)t^\prime \right)$ (see, Section \ref{pdfTn}). 

\noindent
Note that \eqref{pf1} provides a pricing formula driven by a mixed BG process, which extends the pricing formula, for the case $n\geq 2$, for BG stock models considered in Proposition 8.3 of \cite{bilateral0}.

\noindent
 As a special case, we get an arbitrage-free pricing formula driven by the L\'evy process associated with the linear combination of gamma r.v.s as:
\begin{align}\label{lcgop}
\pi
&=c_ne^{-rT}\sum_{j=0}^{\infty}\gamma_j \frac{\eta^{(p+j)t^\prime }  }{\Gamma t^\prime (p+j)}\int_{\ln\left(\frac{K}{s}\right)}^{\infty} (se^x - K)x^{(p+j)t^\prime-1}e^{-\eta x}~ dx,
\end{align}
\noindent
where $\eta=\alpha/1-\alpha$, and $c_n,\gamma_j,p, \alpha$ are defined in \eqref{ref1} and \eqref{ref2}. Let $\Gamma(w,z):=\int_{z}^{\infty}x^{w-1}e^{-x}dx,$ $~z\in \mathbb{R}$ be the incomplete gamma function. Then, \eqref{lcgop} can be written as
\begin{align*}
\pi
&=c_ne^{-rT}\sum_{j=0}^{\infty} \frac{\gamma_j   }{\Gamma  (p+j)t^\prime}\bigg( s \left(\frac{\eta}{\eta -1}\right)^{(p+j)t^\prime}  \Gamma \left((p+j)t^\prime, (\eta-1)ln\left(\frac{K}{s}\right)\right) \\
&\quad\quad-  K\Gamma \left((p+j)t^\prime, \eta ln\left(\frac{K}{s}\right)\right) \bigg).
\end{align*}
\noindent
 When $s=K$, we get an exact pricing formula as
\begin{align*}
\pi&=c_nKe^{-rT}\sum_{j=0}^{\infty}\gamma_j \frac{\eta^{(p+j)t^\prime }  }{\Gamma t^\prime (p+j)}\int_{0}^{\infty}e^xx^{(p+j)t^\prime-1}e^{-\eta x}~ dx -Ke^{-rT}\\
&=Ke^{-rT}\left(c_n\sum_{j=0}^{\infty}\gamma_j\left( \frac{\eta}{\eta -1} \right)^{(p+j)t^\prime }-1\right)\\
&=Ke^{-rT}\left( \mathbb{E}_{L_n} \left( \frac{\eta}{\eta -1} \right)^{(L_n+p)t^\prime }-1\right),
\end{align*}
\noindent
where $\mathbb{E}_{L_n}$ is the expectation with respect to the rv $L_n$.



\vskip 1ex
\noindent
\textbf{Acknowledgement:} The first author acknowledges financial support from the Research Seed Grant of NIT Warangal.

\setstretch{1}

\end{document}